\documentclass[a4paper,reqno]{amsart}

\usepackage[foot]{amsaddr}

\usepackage{amsmath,amssymb,amsfonts,amsthm}
\usepackage[colorlinks=true,citecolor=blue,hypertexnames=false]{hyperref}
\usepackage[utf8]{inputenc}
\usepackage[american]{babel}
\usepackage{enumerate}
\usepackage{tikz,calc}
\usetikzlibrary{decorations.pathreplacing,arrows}
\tikzset{>=stealth}
\usepackage{multicol}
\usepackage{bm}
\usepackage{float}
\usepackage{booktabs}
\floatstyle{ruled}
\newfloat{algorithm}{htb}{alg}
\floatname{algorithm}{Algorithm}
\usepackage[section]{placeins}
\usepackage{cleveref}
\usepackage{siunitx}

\usepackage[
maxbibnames=99,
style=ext-numeric,
backend=bibtex,
articlein=false,
sorting=nyt,
sortcites,
natbib=false,
doi=false,
url=false,
isbn=false,
eprint=false]{biblatex}

\DefineBibliographyStrings{american}{phdthesis = {Ph.D. thesis}}

\makeatletter
\newcommand{\neutralize}[1]{\expandafter\let\csname c@#1\endcsname\count@}
\makeatother

\theoremstyle{plain}
\newtheorem{theorem}{Theorem}

\newtheorem{corollary}{Corollary}
\newtheorem{proposition}{Proposition}
\newtheorem{lemma}{Lemma}[section]
\newtheorem{claim}{Claim}
\newtheorem{conjecture}{Conjecture}

\theoremstyle{definition}
\newtheorem{definition}{Definition}

\newtheorem{example}{Example}
\newtheorem{observation}{Observation}
\newtheorem{remark}{Remark}
\newtheorem{problem}{Problem}
\newtheorem{question}{Question}

\newcommand{\nats}{\mathbb{N}}

\newcommand{\F}{\mathcal{F}}

\newcommand{\vect}[1]{\bm{#1}}
\newcommand{\abs}[1]{\left\lvert#1\right\rvert}

\let\leq\leqslant
\let\geq\geqslant
\let\phi\varphi
\let\le\leq
\let\ge\geq

\title{The Saturation Spectrum for Antichains of Subsets}

\author[J.R. Griggs]{Jerrold R. Griggs$^1$}
\thanks{The research of Jerrold Griggs is supported in part by grant \#282896 from the Simons Foundation.}
\author[T. Kalinowski]{Thomas Kalinowski$^{2,5}$}
\author[U. Leck]{Uwe Leck$^3$}
\author[I.T. Roberts]{Ian T. Roberts$^4$}
\author[M. Schmitz]{Michael Schmitz$^3$}

\address{$^1$Department of Mathematics, University of South Carolina, Columbia, SC  USA }
\address{$^2$Fakult\"at f\"ur angewandte Computer- und Biowissenschaften, Hochschule Mittweida, Germany}
\address{$^3$Institut f\"ur Mathematik, Europa-Universit\"at Flensburg, Germany}
\address{$^4$College of Education, Charles Darwin University, Darwin, NT, Australia}
\address{$^5$School of Science and Technology, University of New England, Armidale, NSW, Australia}

\email[J.R.~Griggs]{griggs@math.sc.edu}
\email[T.~Kalinowski]{kalinows@hs-mittweida.de}
\email[U.~Leck]{Uwe.Leck@uni-flensburg.de}
\email[I.T.~Roberts]{Ian.Roberts@cdu.edu.au}
\email[M.~Schmitz]{Michael.Schmitz@uni-flensburg.de}

\date{\today}

\begin{document}

\begin{abstract}
  Extending a classical theorem of Sperner, we characterize the integers $m$ such that there exists
  a maximal antichain of size $m$ in the Boolean lattice $B_n$, that is, the power set of
  $[n]:=\{1,2,\dots,n\}$, ordered by inclusion. As an important ingredient in the proof, we initiate
  the study of an extension of the Kruskal-Katona theorem which is of independent interest. For
  given positive integers $t$ and $k$, we ask which integers $s$ have the property that there exists
  a family $\mathcal F$ of $k$-sets with $\abs{\mathcal F}=t$ such that the shadow of $\mathcal F$
  has size $s$, where the shadow of $\mathcal F$ is the collection of $(k-1)$-sets that are
  contained in at least one member of $\mathcal F$. We provide a complete answer for $t\leq
  k+1$. Moreover, we prove that the largest integer which is not the shadow size of any family of
  $k$-sets is $\sqrt 2k^{3/2}+\sqrt[4]{8}k^{5/4}+O(k)$.
\end{abstract}

\maketitle

\section{Introduction}\label{sec:intro}

In the Boolean lattice $B_n$ of all subsets of $[n]:=\{1,2,\dots,n\}$, ordered by inclusion, an
\emph{antichain} is a family of subsets such that no one contains any other one. By a classical
theorem of Sperner~\cite{Sperner1928}, the maximum size of an antichain in $B_n$ is
$\binom{n}{\lfloor n/2 \rfloor}$, and it is obtained only by the collections
$\binom{[n]}{\lfloor n/2\rfloor}$ and $\binom{[n]}{\lceil n/2\rceil}$, where for any set $M$ we write
$\binom{M}{k}$ for the collection of all $k$-element subsets of $M$.

Of course there are antichains in $B_n$ of every size $s$ up to $\binom{n}{\lfloor n/2 \rfloor}$,
since any $s$ subsets of size $\lfloor n/2\rfloor$ will do.  Trotter (pers. comm. 2016) wondered
what is the second largest size of a \emph{maximal} antichain of subsets in $B_n$?  Here, an
antichain $\mathcal A$ is maximal if no subset can be added to it without destroying the antichain
property. Looking at examples, a natural candidate for $\mathcal A$ is obtained by taking the
collection $\binom{[n]}{\lceil n/2\rceil}$, add in the set $\{1,\dots,{\lceil n/2\rceil}-1\}$, and
then remove the sets above it. Using methods in the field, one can show that this is indeed a
maximal antichain with the second largest size, which is
$\binom{n}{\lfloor n/2 \rfloor}-{\lfloor n/2 \rfloor}$.  Notice the gap between the first and second
largest sizes of maximal antichains in $B_n$.  What can be said about, say, the third largest size?
Indeed, what can we say in general about the possible sizes of maximal antichains in $B_n$?  Our
group set out to investigate this natural question, and we found that things get increasingly
complicated the closer we look.  We can now describe where the gaps are (intervals of values for
which there is no maximal antichain), and construct antichains for all sizes below the gaps, thus
answering the following question.
\begin{question}\label{q:MAC_sizes}
  For which integers $m$ with $1\leq m\leq\binom{n}{\lfloor n/2\rfloor}$ does there exist a maximal
  antichain of size $m$?
\end{question}
In other words, we determine the set
\[S(n)=\left\{\abs{\mathcal A}\,:\,\mathcal A\text{ is a maximal antichain in }B_n\right\}.\] If
$\mathcal A$ contains only sets of a single size $k$, then we must take all of them to get a maximal
antichain, and this gives us $\binom{n}{k}\in S(n)$ for $k=1,2,\dots,\lfloor n/2\rfloor$. As a next
step we consider maximal antichain sizes obtained from antichains $\mathcal A$ that are contained in
\emph{two} consecutive levels of $B_n$, that is,
$\mathcal A\subseteq\binom{[n]}{k}\cup\binom{[n]}{k-1}$, for some $k$. Such antichains are called
\emph{flat}, and we are also interested in the following more restricted version of
Question~\ref{q:MAC_sizes}.
\begin{question}\label{q:MFAC_sizes}
  For which integers $m$ with $1\leq m\leq\binom{n}{\lfloor n/2\rfloor}$ does there exist a maximal flat
  antichain of size $m$?
\end{question}
In the present paper we make a first step towards answering this question by establishing that the
sizes close to the maximum $\binom{n}{\lfloor n/2\rfloor}$ are achieved by flat maximal
antichains. This is complemented in~\cite{Griggs2021c} by showing that all the sizes, with some
exceptions below $\binom{n}{2}$, can be obtained by flat antichains, thus completing the answer to
Question~\ref{q:MFAC_sizes}.

These investigations are analogous to the study of the \emph{saturation spectrum} in graphs: For a
family $\mathcal H$ of graphs, one is interested in the possible edge numbers of saturated
$\mathcal H$-free graphs on $n$ vertices, where saturation means that adding any edge creates a copy
of a member of $\mathcal H$. Starting with \cite{Barefoot_1995}, this natural extension of Tur\'an
type problems (asking for the maximal number of edges in an $\mathcal H$-free graph) has been
studied for various families $\mathcal H$, see~\cite{Gould2019,Faudree2011} for overviews. More recently,
similar saturation problems have been studied for posets in $B_n$: For a given poset $P$, what are
the possible sizes of saturated $P$-free families $\mathcal F\subseteq 2^{[n]}$, that is, such that
no subset can be added to $\mathcal F$ without creating a copy of $P$?
In~\cite{Morrison_2014,Gerbner2012,Keszegh_2021} the focus is on the lower end of the saturation
spectrum, that is, the minimum size of a saturated $P$-free family, while the classical Sperner
theory looks at the upper end of the spectrum, that is, the maximum size of a $P$-free family. A
natural next step is to extend this by asking for a characterization of all possible sizes of
saturated $P$-free families. Our work on maximal antichains can be seen as a first step in this
direction by looking at the special case where $P$ is a chain of length 2. Another related direction
is the study of stability questions: What can be said about the structure of a $P$-free family whose
size is close to the maximum possible? For instance, in~\cite{Patkos2015} a stability result for $P$
being a chain of length 3 has been used in connection with supersaturation results for the butterfly
poset. Stability results are related to our problem as these structural statements can imply gaps at
the upper end of the spectrum, and our arguments in \Cref{sec:MAC_sizes} have this flavor.

As an intermediate step towards characterizing the maximal antichain sizes very close to
$\binom{n}{\lfloor n/2\rfloor}$, we investigate a problem related to another classical result, known
as the Kruskal-Katona Theorem~\cite{Kruskal1963,Katona1968}. This theorem answers the following
question: Given positive integers $t$ and $k$, how can one select a family $\F$ of $t$ $k$-sets in
order to minimize the size of the family $\Delta\F$ of $(k-1)$-subsets that are each contained in
some set in $\F$? The family $\Delta\F$ is called the \emph{shadow} of $\F$. It is then natural to
ask the following question.
\begin{question}\label{q:shadow_spectrum}
  Given $t$ and $k$, what is the \emph{shadow spectrum}
  \[\sigma(t,k):=\left\{\abs{\Delta\mathcal F}\,:\,\mathcal
  F\text{ is a family of $k$-sets with }\abs{\mathcal F}=t\right\}\,\text{?}\]
\end{question}
The Kruskal-Katona Theorem characterizes the smallest element of $\sigma(t,k)$, and the largest
element is obviously $tk$. \Cref{q:MFAC_sizes} and \Cref{q:shadow_spectrum} are closely related,
since a maximal flat antichain $\mathcal A$ which consists of $k$-sets and $(k-1)$-sets is
determined by its collection of $k$-sets, $\mathcal F:=\mathcal A\cap\binom{[n]}{k}$, namely
$\mathcal A=\mathcal F\cup\binom{[n]}{k-1}\setminus\Delta\mathcal F$. However, it is important to
note that an antichain of this form is not necessarily maximal, as there can be $k$-sets not in
$\mathcal F$ which are not supersets of any $(k-1)$-set outside $\Delta\mathcal F$. To be more
precise, we introduce the notion of the \emph{shade} (or \emph{upper shadow}) of a family
$\mathcal{G} \subseteq \binom{[n]}{\ell}$ which is defined to be the family
$\nabla \mathcal{G} = \{G \cup \{x\} \,:\, G \in \mathcal{G},~ x \in [n] \setminus G\}$. Then
$\mathcal A=\mathcal F\cup\mathcal G$ with $\mathcal F\subseteq\binom{[n]}{k}$ and
$\mathcal G\subseteq\binom{[n]}{k-1}$ is a maximal antichain if and only if
$\mathcal G=\binom{[n]}{k-1}\setminus\Delta\mathcal F$ and
$\mathcal F=\binom{[n]}{k}\setminus\nabla\mathcal G$.

\subsection*{Related work} The paper~\cite{Gruettmueller2009} asks for the minimum size of a maximal
flat antichain for given $n$ and $k$. The case $k=3$ is settled by showing that the minimum size of
a maximal flat antichain consisting of 2-sets and 3-sets is $\binom{n}{2}-\lfloor (n+1)^2/8\rfloor$,
and all extremal antichains are determined. A stability version of this result is proved
in~\cite{Gerbner2012}, together with bounds for general $k$ which depend on the Tur\'an densities
for complete $k$-uniform hypergraphs. For $k\geq 4$ these bounds are very far away from the best
known constructions, and a lot of work remains to be done. Further results are given
in~\cite{Kalinowski2013} for the general setting of maximal antichains of subsets, where set sizes
belong to some given set $K$. The flat case corresponds to $K=\{k,k-1\}$. In~\cite{Griggs2021} the
focus is on (not necessarily maximal) antichains of the form
$\mathcal F\cup\binom{[n]}{k-1}\setminus\Delta\mathcal F$, where $\mathcal F$ is a family that
minimizes the shadow size among all families of $\abs{\mathcal F}$ $k$-sets. It has been observed
several times, for instance in~\cite{Frankl1995}, that the size
$\abs{\mathcal F\cup\binom{[n]}{k-1}\setminus\Delta\mathcal F}$ is a rather complicated function of
$\abs{\mathcal F}$, and~\cite{Griggs2021} studies the structure of (weighted variants of) this
function. The sets $\sigma(t,k)$ were first considered in~\cite{Leck1995}, where certain classes of
triples $(s,t,k)$ with $s\not\in\sigma(t,k)$ were characterized.

\subsection*{Main Results}\label{subsec:contributions}

\begin{theorem}\label{thm:main_result}
  Let $n$ and $m$ be positive integers with $m \le \binom{n}{\lceil n/2\rceil}$.
  \begin{enumerate}[(i)]
  \item For $\displaystyle m>\binom{n}{\lceil n/2\rceil}-\left\lceil\frac{n}{2}\right\rceil^2$, there exists a maximal antichain of
    size $m$ if and only if
    \[m=\binom{n}{\lceil n/2\rceil}-tl+\binom{a}{2}+\binom{b}{2}+c\] for some integers $l\in\{\lceil n/2\rceil,\lfloor n/2 \rfloor\}$,
    $t\in\{0,\dots,\lceil n/2\rceil\}$, $a\geq b\geq c\geq 0$, $1\leq a+b\leq t$. Moreover, every such $m$ is
    obtained as the size of a maximal flat antichain.
  \item If
    $\displaystyle m \le \binom{n}{\lceil
      n/2\rceil}-\left\lceil\frac{n}{2}\right\rceil\left\lceil\frac{n+2}{4}\right\rceil$, then there
    is a maximal antichain of size $m$ in $B_n$.
  \end{enumerate}
\end{theorem}

The last statement in part (i) of this theorem is an initial step towards an answer for
Question~\ref{q:MFAC_sizes}: The sizes close to the maximum are already obtained by flat maximal
antichains, and in \cite{Griggs2021c} we can focus on the case $m\leq\binom{n}{\lceil n/2\rceil}-\left\lceil\frac{n}{2}\right\rceil^2$.

One direction of the equivalence in \Cref{thm:main_result}(i) will be proved by
showing that, for any $n,l,t,a,b,c$ as in the theorem and $k=\lceil n/2\rceil$, we can obtain a maximal antichain
$\mathcal A\subseteq\binom{[n]}{l}\cup\binom{[n]}{l+1}$ of size
$\abs{\mathcal A}=\binom{n}{k}-tl+\binom{a}{2}+\binom{b}{2}+c$ by taking a suitable family
$\mathcal F\subseteq\binom{[n]}{l+1}$ with $\abs{\mathcal F}=t$ and setting
$\mathcal A=\mathcal F\cup\left(\binom{[n]}{l}\setminus\Delta\mathcal F\right)$. To establish this
result, we study the sets $\sigma(t,k)$. Our proof of part $(i)$ of Theorem~\ref{thm:main_result} in
Section~\ref{sec:MAC_sizes} is based on the following characterization of the sets $\sigma(t,k)$ for
$t\leq k+1$, which will be established in Section~\ref{sec:shadow_spectrum}.
\begin{theorem}\label{thm:shadow_spectrum}
  For integers $k$ and $t$ with $k\geq 2$ and $1\leq t\leq k+1$,
  \begin{equation}\label{eq:shadow_spectrum}
  \sigma(t,k)=\left\{tk-\binom{a}{2}-\binom{b}{2}-c\,:\,a\geq b\geq c\geq
  0,\,1\leq a+b\leq t\right\}.
  \end{equation}
  Moreover, for every $s\in\sigma(t,k)$, there exists a family $\mathcal F\subseteq\binom{[k+4]}{k}$
  with $\abs{\mathcal F}=t$ and $\abs{\Delta\mathcal F}=s$.
\end{theorem}
To get a better idea of the structure of the sets $\sigma(t,k)$, and also as a useful tool in the
proof of \Cref{thm:shadow_spectrum}, we write the right-hand side of~(\ref{eq:shadow_spectrum})
as a union of pairwise disjoint intervals. To state this precisely, set
\begin{equation}\label{eq:def_It}
  I(t)=\left\{\binom{a}{2}+\binom{b}{2}+c\,:\,a\geq b\geq c\geq 0,\,1\leq a+b\leq t\right\},
\end{equation}
so that \Cref{thm:shadow_spectrum} says $\sigma(t,k)=\{tk-x\,:\,x\in I(t)\}$. For a positive integer
$t$, we define $j^*(t)$ to be the smallest non-negative integer $j$ with $\binom{j+3}{2}\geq t$. An
easy calculation (see Lemma~\ref{lem:jstar} below) shows that
$j^*(t)=\left\lceil\sqrt{2t}-5/2\right\rceil$. We define intervals
\begin{align*}
  I_j(t) &= \left[\binom{t-j}{2},\,\binom{t-j}{2}+\binom{j+1}{2}\right]&&\text{for }j=0,1,\dots,j^*(t)-1,\\
  I_j(t) &= \left[0,\,\binom{t-j}{2}+\binom{j+1}{2}\right] &&\text{for }j=j^*(t).
\end{align*}
The set $I(t)$ is the union of the intervals $I_j(t)$.
\begin{proposition}\label{prop:interval_description}
  For every positive integer $t$, $\displaystyle I(t)=\bigcup_{j=0}^{j^*(t)}I_j(t)$.
\end{proposition}
As a consequence, we obtain $\sigma(t,k)$ as a union of intervals, the rightmost of them ending at
$tk$. This is illustrated in the following example.
\begin{example}\label{ex:Sigma}
  \Cref{fig:Sigma_50} shows the sets $\sigma(t,50)$ for $t=1,\dots,14$.
  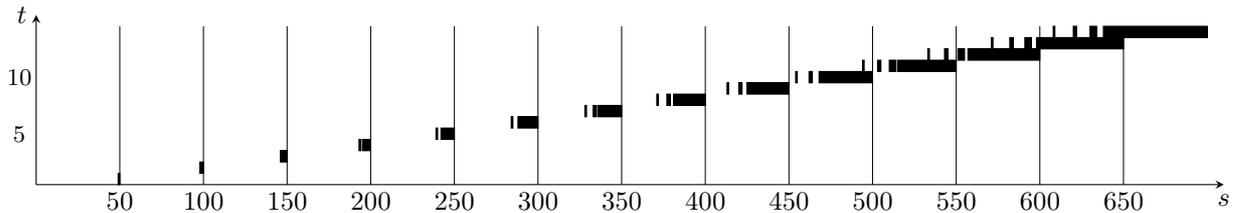
\begin{figure}[htb]
    \centering
    \begin{tikzpicture}[xscale=.22,yscale=.15]
      \draw[->] (0,0) -- (0,15) node[anchor=east] {$t$};
      \draw[->] (0,0) -- (71,0) node[anchor=north] {$s$};
      \foreach \i in {1,...,13}
      {
        \draw[very thin] (5*\i,0) -- (5*\i,14);
        \pgfmathtruncatemacro{\label}{50*\i}
        \node at (5*\i,-1.5) {$\label$};
      }
      \node at (-1,4.5) {{\small $5$}};
      \node at (-1,9.5) {{\small $10$}};
      \draw[fill] (4.9,0) rectangle (5,1);
      \draw[fill] (9.8,1) rectangle (10,2);
      \draw[fill] (14.6,2) rectangle (15,3);
      \draw[fill] (19.3,3) rectangle (19.4,4);
      \draw[fill] (19.5,3) rectangle (20,4);
      \draw[fill] (23.9,4) rectangle (24,5);
      \draw[fill] (24.2,4) rectangle (25,5);
      \draw[fill] (28.4,5) rectangle (28.5,6);
      \draw[fill] (28.8,5) rectangle (30,6);
      \draw[fill] (32.8,6) rectangle (32.9,7);
      \draw[fill] (33.3,6) rectangle (33.5,7);
      \draw[fill] (33.6,6) rectangle (35,7);
      \draw[fill] (37.1,7) rectangle (37.2,8);
      \draw[fill] (37.7,7) rectangle (37.9,8);
      \draw[fill] (38.1,7) rectangle (40,8);
      \draw[fill] (41.3,8) rectangle (41.4,9);
      \draw[fill] (42,8) rectangle (42.2,9);
      \draw[fill] (42.5,8) rectangle (45,9);
      \draw[fill] (45.4,9) rectangle (45.5,10);
      \draw[fill] (46.2,9) rectangle (46.4,10);
      \draw[fill] (46.8,9) rectangle (50,10);
      \draw[fill] (49.4,10) rectangle (49.5,11);
      \draw[fill] (50.3,10) rectangle (50.5,11);
      \draw[fill] (51,10) rectangle (51.4,11);
      \draw[fill] (51.5,10) rectangle (55,11);
      \draw[fill] (53.3,11) rectangle (53.4,12);
      \draw[fill] (54.3,11) rectangle (54.5,12);
      \draw[fill] (55.1,11) rectangle (55.5,12);
      \draw[fill] (55.7,11) rectangle (60,12);
      \draw[fill] (57.1,12) rectangle (57.2,13);
      \draw[fill] (58.2,12) rectangle (58.4,13);
      \draw[fill] (59.1,12) rectangle (59.5,13);
      \draw[fill] (59.8,12) rectangle (65,13);
      \draw[fill] (60.8,13) rectangle (60.9,14);
      \draw[fill] (62,13) rectangle (62.2,14);
      \draw[fill] (63,13) rectangle (63.4,14);
      \draw[fill] (63.8,13) rectangle (70,14);
     \end{tikzpicture}
    \caption{The sets $\sigma(1,50),\sigma(2,50),\dots,\sigma(14,50)$.}
    \label{fig:Sigma_50}
  \end{figure}
\end{example}

In our last result, we investigate the largest integer which is not the shadow size of a $k$-uniform
set family, that is, the quantity $\psi(k)=\max\nats\setminus\Sigma(k)$ where
$\Sigma(k)=\bigcup_{t=1}^{\infty}\sigma(t,k)$. For instance, \Cref{fig:Sigma_50} suggests
$\psi(50)=12\times 50-\binom{12-3}{2}-\binom{4}{2}-1=557$. In the next theorem, we provide an
asymptotic formula for $\psi(k)$, and combining this with Theorem~\ref{thm:main_result}, we obtain an
asymptotic formula for the smallest positive integer which is not the size of a maximal antichain, that
is, for $\phi(n)=\min\left\{m\in\nats\,:\,m\not\in S(n)\right\}$.

\begin{theorem}\label{thm:asymptotics}\hfill
  \begin{enumerate}[(i)]
  \item\label{asymptotics:item_i} For $k\to\infty$, $\displaystyle\psi(k)=\sqrt 2k^{3/2}+\sqrt[4]{8}k^{5/4}+O(k)$.
  \item\label{asymptotics:item_ii} For $n\to\infty$, $\displaystyle\phi(n)=\binom{n}{\lfloor n/2\rfloor}-\left(\frac12+o(1)\right)n^{3/2}$.
  \end{enumerate}
\end{theorem}

\subsection*{Outline of the paper}\label{subsec:outline}
\begin{enumerate}
\item In \Cref{sec:shadow_spectrum}, we prove \Cref{thm:shadow_spectrum} and
  \Cref{prop:interval_description}. Using a construction based on certain graphs we show
  $\{tk-x\,:\,x\in I(t)\}\subseteq\sigma(t,k)$ (\Cref{lem:subset}), then we prove
  \Cref{prop:interval_description}, and use it to establish
  $\sigma(t,k)\subseteq\{tk-x\,:\,x\in I(t)\}$ (\Cref{lem:superset}).
\item The proof of \Cref{thm:main_result}(i) is the content of \Cref{sec:MAC_sizes}. The existence of
  maximal antichains with the claimed sizes is derived from \Cref{thm:shadow_spectrum} using the
  construction described just after the statement of \Cref{thm:main_result}, and this is
  complemented by showing that all antichains of the relevant sizes necessarily come from this
  construction.
\item \Cref{sec:lower_part} contains the proof of \Cref{thm:main_result}(ii).
\item The asymptotic results in \Cref{thm:asymptotics} are derived in \Cref{sec:asymptotics}, and we
  conclude by listing a few open problems in \Cref{sec:open}.
\end{enumerate}

A common theme in Sections~\ref{sec:MAC_sizes} and~\ref{sec:lower_part} is the construction of
families of maximal antichains whose sizes form intervals, and then establishing that these
intervals overlap to yield all the required sizes. Both steps tend to get easier as $n$ grows: For
small $n$, the general constructions leave a few gaps, and even when there are no gaps it turns out
that the straightforward arguments for the inequalities, which verify the overlaps, sometimes only work
for, say, $n\geq 200$. In both situations one can try to optimize the constructions and
the proofs to work for smaller $n$. As this would lead to a further blow-up of various case
distinctions and computations (on which the present version is already rather heavy), we decided to
take the easy way out, and treat the small cases by ad-hoc constructions and inequality
verifications (either by hand or using a computer).

\section{The shadow spectrum}\label{sec:shadow_spectrum}
The aim of this section is to prove \Cref{thm:shadow_spectrum}. To this end we will construct families of
$k$-sets by taking the complements of the edges of certain graphs whose vertex sets are contained in
$[k+2]$. This class of graphs is defined as follows.

\begin{definition}\label{def:G(a,b,c)}
  Let $a,b,c$ be integers with $a \ge b \ge c \ge 0$ and $a \ge 1$. For $b=0$ let $G(a,b,c)$ be an
  $a$-star, that is, a complete bipartite graph $K_{1,a}$, on the vertex set $[a+1]$. For $b \ge 1$
  let $G(a,b,c)$ be the graph on the ground set $[a+b-c+2]$ that consists of one $a$-star and one
  $b$-star that share exactly $c$ pendant vertices (see \Cref{fig:G(abc)}).
\end{definition}
\begin{figure}[htb]
	\centering
	\begin{tikzpicture}[scale=1.5,every node/.style={draw,circle,fill,outer sep=1pt,inner
			sep=1pt}]
		\node (a) at (0,1.5) {};
		\node (b) at (3,1.5) {};
		\node (v1) at (-.5,0) {};
		\node[draw=none,fill=none] at (0,0) {$\ldots$};
		\node (v2) at (.5,0) {};
		\node (v3) at (1,0) {};
		\node[draw=none,fill=none] at (1.5,0) {$\ldots$};
		\node (v4) at (2,0) {};
		\node (v5) at (2.5,0) {};
		\node[draw=none,fill=none] at (3,0) {$\ldots$};
		\node (v6) at (3.5,0) {};
		\draw[thick] (a) -- (v1);
		\draw[thick] (a) -- (v2);
		\draw[thick] (a) -- (v3);
		\draw[thick] (a) -- (v4);
		\draw[thick] (b) -- (v3);
		\draw[thick] (b) -- (v4);
		\draw[thick] (b) -- (v5);
		\draw[thick] (b) -- (v6);
		\draw [thick,decoration={brace,mirror,raise=0.1cm},decorate] (.95,0) -- (2.05,0)
		node [fill=none,draw=none,pos=0.5,anchor=north,yshift=-0.2cm] {$c$};
		\draw[thick,domain=-120:-25] plot ({0.4*cos(\x)},{1.5+0.4*sin(\x)});
		\node[fill=none,draw=none] at (-.3,1.2) {$a$};
		\draw[thick,domain=-150:-60] plot ({3+0.4*cos(\x)},{1.5+0.4*sin(\x)});
		\node[fill=none,draw=none] at (3.3,1.2) {$b$};
	\end{tikzpicture}
	\caption{The graph $G(a,b,c)$.}\label{fig:G(abc)}
\end{figure}
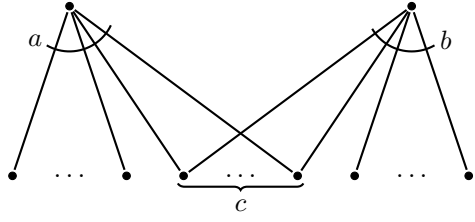

The graph $G(a,b,c)$ has $a+b$ edges and $\binom{a}{2}+\binom{b}{2}+c$ pairs of adjacent edges.

\begin{definition}\label{def:F(a,b,c)}
  For integers $k$ and $a, b, c$ with $a\ge b \ge c \ge 0$, let $\mathcal{F}_k(a,b,c)$ be the family
  \[\mathcal{F}_k(a,b,c) = \left\{[k+2]\setminus e\,:\, e \text{ is an edge of } G(a,b,c)\right\}.\]
\end{definition}

\begin{remark}\label{re:F(abc)}
  If $b=0$ and $k\geq a-1$ or $b\geq 1$ and $k\geq a+b-c$ then the vertex set of $G(a,b,c)$ is
  contained in $[k+2]$, and consequently, $\mathcal F_k(a,b,c)$ is a family of $k$-sets with
  $\abs{\mathcal{F}_k(a,b,c)} = a+b$. Moreover, in this situation, the size of its shadow equals its own
  size multiplied by $k$ reduced by the number of pairs of sets that have a shadow element in
  common. As the latter equals the number of pairs of adjacent edges of $G(a,b,c)$, we obtain
  \[\abs{\Delta \mathcal{F}_k(a,b,c)} = (a+b)k - \binom{a}{2} - \binom{b}{2} - c.\]
\end{remark}

Before we present rigorous proofs for every detail that is needed to obtain
\Cref{thm:shadow_spectrum}, we sketch how the different expressions in~(\ref{eq:shadow_spectrum})
arise and how they are related. A $t$-family $\mathcal{F}=\{A_1, \ldots, A_t\}$ of $k$-sets can be
regarded as built by successively adding its elements $A_1, \ldots, A_t$ in this order. Its shadow
size $s=|\Delta \mathcal{F}|$ is the sum of the \emph{marginal shadow sizes} (or
\emph{new-shadow sizes})
\[s_i= \abs{\Delta \{A_1,\dots,A_i\}} - \abs{\Delta \{A_1,\dots,A_{i-1}\}} \quad \text{ for } i = 1,
  2, 3, \dots, t.\] We let $\vect{s} = (s_1,\dots, s_t)$ be the so called \textit{marginal shadow
  vector} (or \emph{new-shadow vector}). By the Kruskal-Katona Theorem, the shadow size
$s=\sum_{i=1}^t s_i$ is minimal if $\mathcal{F}$ consists of the first $t$ elements of
$\binom{\nats}{k}$ in \emph{squashed order} $<_S$, which is defined by
\[A <_S B \iff \max\left[(A\setminus B) \cup (B\setminus A)\right] \in B.\]
For more on squashed
order and the Kruskal-Katona Theorem see e.g. \parencite[][pp. 112-124]{Anderson2002}. For
$2 \le t\le k+1$ taking the first $t$ elements in squashed order means choosing only subsets of
$[k+1]$. The marginal shadow vector is then $\vect{s} = (k, k-1, \ldots, k-(t-1))$ and the shadow size
is given by $s=tk-\sum_{i=1}^{t-1} i = tk- \binom{t}{2}$.

How can one increase this minimal shadow size by the least possible amount? Clearly, by introducing
exactly one new element, say $k+2$, within the last set $A_t$. Then the marginal shadow vector
becomes $\vect{s}=(k,k-1, \ldots, k-(t-2);k-1)$, where the semicolon indicates at which point we first
used a set that is not a subset of $[k+1]$. The resulting shadow size is
$s=tk - \binom{t-1}{2} - \binom{2}{2}$. By introducing two new elements instead of one within $A_t$
we can get the marginal shadow vector $\vect{s}= (k,k-1, \dots, k-(t-2);k)$ and the shadow size
$s=tk - \binom{t-1}{2}$.

In order to increase the shadow size again by the least possible amount, we have to introduce a new
element one step earlier, that is within $A_{t-1}$. From there we proceed in squashed order and
obtain the marginal shadow vector $\vect{s}=(k,k-1, \ldots, k-(t-3);k-1,k-2)$, and the shadow size
$s=tk - \binom{t-2}{2} - \binom{3}{2}$. It is easy to see in this special case and will be shown in
general later that we can increase the shadow size from this value in single steps up to
$s=tk-\binom{t-2}{2}$, which corresponds to the marginal shadow vector
$\vect{s}=(k,k-1, \ldots, k-(t-3);k,k)$.

The next increase by the least possible amount is achieved by introducing a new element already
within $A_{t-2}$ and results in $\vect{s}=(k,k-1, \dots, k-(t-4);k-1,k-2,k-3)$, having the sum
$s=tk - \binom{t-3}{2} - \binom{4}{2}$. Again, we can increase the shadow size from this value in
single steps up to $s=tk-\binom{t-3}{2}$, corresponding to the marginal shadow vector
$\vect{s}=(k,k-1, \dots, k-(t-4);k,k,k)$, and so on.

The general scheme is having a marginal shadow vector of length $t$, that is split up into one
(longer) part of length $t-j$ and one (shorter) part of length $j$, and has the form
\[\vect{s}=(k,k-1, \ldots, k-(t-j-1);k-1,k-2,\ldots,k-j),\]
and the sum $s=tk - \binom{t-j}{2} - \binom{j+1}{2}$. From here we can increase the shadow size in
single steps up to the marginal shadow vector $\vect{s}=(k,k-1, \ldots, k-(t-j-1);k,\ldots,k)$, and
thereby we obtain all shadow sizes in the interval
\[tk-I_j(t) =\left\{tk-x\,:\,\binom{t-j}{2}\leq x\leq\binom{t-j}{2}+\binom{j+1}{2}\right\}.\]
Then comes a
gap before we get all shadow sizes in $tk-I_{j+1}(t)$, and so on. It is easily seen that the gaps
between the intervals $I_j(t)$ and $I_{j+1}(t)$ close at some point for relatively small $j$. This happens
exactly when
\[\binom{t-j-1}{2} + \binom{j+2}{2} + 1 \ge \binom{t-j}{2},\]
and a simple computation yields that this is the case if and only if $\binom{j+3}{2} \ge
t$. Therefore, we chose $j^*(t)$ to be the smallest non-negative integer with this property.

How is this related to the graphs $G(a,b,c)$? We have seen that it is crucial to increase the shadow
size from $s=tk - \binom{t-j}{2} - \binom{j+1}{2}$ up to $s=tk - \binom{t-j}{2}$ in single
steps. Therefore, we rewrite
\[s=tk - \binom{t-j}{2} - \binom{j+1}{2}=tk - \binom{t-j}{2} - \binom{j}{2}  -j,\]
and this has the form $tk-\binom{a}{2} - \binom{b}{2} - b$ with $a\ge b\ge 0$. This shadow size corresponds to
$G(a,b,b)$, or more precisely to $\mathcal{F}_k(a,b,b)$ extended by $t-(a+b)$ sets with marginal
shadow $k$, because this graph has exactly $\binom{a}{2} + \binom{b}{2} + b$ pairs of adjacent
edges. To increase the shadow size in single steps we use the graphs $G(a,b,b-1), G(a,b,b-2),
\dots, G(a,b,0)$ with $\binom{a}{2} + \binom{b}{2} + b-1,~ \binom{a}{2} + \binom{b}{2} + b-2,
\dots, \binom{a}{2} + \binom{b}{2}$ pairs of adjacent edges. These are only $b=j$ single steps, but
we can use an inductive argument for the missing steps later.

\medskip

Now we turn to the details. We start by using the families $\mathcal{F}_k(a,b,c)$ to prove that the
right-hand side of~(\ref{eq:shadow_spectrum}) is contained in the left-hand side.
\begin{lemma}\label{lem:subset}
  For integers $k$ and $t$ with $k\geq 2$ and $1\leq t\leq k+1$, $\left\{tk-x\,:\,x\in I(t)\right\}\subseteq\sigma(t,k)$.
\end{lemma}
\begin{proof}
We fix $s=tk-\binom{a}{2}-\binom{b}{2}-c$ with $a \ge b \ge c \ge 0$ with $1 \le a+b \le t$,
and show that there is a $t$-family of $k$-sets with shadow size $s$.
\begin{description}
\item[Case 1] $a+b=t=k+1$ and $c=0$.
  \begin{description}
  \item[Case 1.1] $b \ge 2$. Then we let $\mathcal{F} = \mathcal{F}_1 \cup \{A\}$, where
    $\mathcal{F}_1 = \mathcal{F}_k(a,b-1,b-1)$, and is $A$ is a $k$-set that is shadow-disjoint to
    all sets in $\mathcal{F}_1$. As the ground set of $\mathcal{F}_1$ is $[a+2]$, which is a subset
    of $[k+2]$, $\mathcal F_1$ is a family of $k$-sets of size $a+b-1$. Using \Cref{re:F(abc)} we obtain
    \[\abs{\Delta \mathcal{F}} = \abs{\Delta \mathcal{F}_1} + k = (a+b-1)k - \binom{a}{2} -
      \binom{b-1}{2} - (b-1) + k = tk - \binom{a}{2} - \binom{b}{2}.\]
\item[Case 1.2] $b \in \{0,1\}$. Note that as $c=0$ we are then looking for a set system with shadow
  size $tk-\binom{a}{2}$. If $b=0$, we have $a=t$ and use $\mathcal{F}=\mathcal{F}_k(a,0,0)$ with
  shadow size $ak- \binom{a}{2} = tk-\binom{a}{2}$. If $b=1$, we use
  $\mathcal{F}= \mathcal{F}_1 \cup \{A\}$, where $\mathcal{F}_1 = \mathcal{F}_k(a,0,0)$ and $A$ is
  shadow disjoint to all members of $\mathcal{F}_1$, and obtain
\[\abs{\Delta \mathcal{F}} = ak - \binom{a}{2} + k = tk-\binom{a}{2}.\]
  \end{description}
\item[Case 2] $a+b < t$ or $t \le k$ or $c\ge 1$. Then, we have $a+b-c\le k$ and we can use
  $\mathcal{F}=\mathcal{F}_1 \cup \mathcal{F}_2$, where $\mathcal{F}_1 = \mathcal{F}_k(a,b,c)$, and
  $\mathcal{F}_2$ consists of $t-(a+b)$ pairwise shadow-disjoint sets such that
  $\Delta \mathcal{F}_1 \cap \Delta \mathcal{F}_2 = \emptyset$. By \Cref{re:F(abc)} we
  obtain
\[\abs{\Delta \mathcal{F}} = (a+b)k - \binom{a}{2} - \binom{b}{2} - c + (t-a-b)k = tk - \binom{a}{2} - \binom{b}{2} - c.\qedhere\]
\end{description}
\end{proof}

In the proof of the other inclusion in~(\ref{eq:shadow_spectrum}) we will
use \Cref{prop:interval_description}, so we first prove this result.

\begin{proof}[Proof of \Cref{prop:interval_description}]
  For every $a\in\{1,2,\dots,t\}$,
  \begin{multline*}
    \bigcup_{b=0}^{\min\{a,t-a\}}\left\{\binom{a}{2}+\binom{b}{2}+c\,:\,0\leq c\leq
      b\right\}=\bigcup_{b=0}^{\min\{a,t-a\}}\left[\binom{a}{2}+\binom{b}{2},\,\binom{a}{2}+\binom{b+1}{2}\right]\\
    =\left[\binom{a}{2},\,\binom{a}{2}+\binom{\min\{a,t-a\} +1}{2}\right],
  \end{multline*}
  and as a consequence,
  \begin{equation}\label{eq:interval_split}
    I(t)=\bigcup_{a=1}^t\left[\binom{a}{2},\,\binom{a}{2}+\binom{\min\{a,t-a\} +1}{2}\right].
  \end{equation}
  We split this union into the intervals for $a\geq t-j^*(t)+1$ and the ones for $a\leq t-j^*(t)$.
  For $a\geq t-j^*(t)+1$, it follows from $2j^*(t)\leq t+2$ that $a\geq t-a$. Hence
  \begin{equation}\label{eq:small_j}
    \left[\binom{a}{2},\,\binom{a}{2}+\binom{\min\{a,t-a\}
        +1}{2}\right]=\left[\binom{a}{2},\,\binom{a}{2}+\binom{t-a
        +1}{2}\right]=I_{t-a}(t).
  \end{equation}
  For $a\leq t-j^*(t)-1$, it follows from $j^*(t)=\min\{j\,:\,\binom{j+3}{2}\geq t\}$ that
  \[\binom{t-a+1}{2}\geq\binom{j^*(t)+2}{2}=\binom{j^*(t)+3}{2}-(j^*(t)+2)\geq t-j^*(t)-2\geq a-1.\]
  Together with $\binom{a+1}{2}\geq a-1$, we obtain that there is no gap between the interval for
  $a$ and the interval for $a+1$:
  \[\binom{a}{2}+\binom{\min\{a,t-a\}+1}{2}\geq\binom{a}{2}+(a-1)=\binom{a+1}{2}-1.\]
  Since the function $a\mapsto\binom{a}{2}+\binom{\min\{a,t-a\}+1}{2}$ is non-decreasing,
  $\binom{a}{2}+\binom{\min\{a,t-a\}+1}{2}\leq\binom{t-j^*(t)}{2}+\binom{j^*(t)+1}{2}$ for all $a$, and we obtain
  \begin{equation}\label{eq:jstar}
    \bigcup_{a=0}^{t-j^*(t)}\left[\binom{a}{2},\,\binom{a}{2}+\binom{\min\{a,t-a\}
      +1}{2}\right]
  =\left[0,\,\binom{t-j^*(t)}{2}+\binom{j^*(t)+1}{2}\right]=I_{j^*(t)}(t).
\end{equation}
Combining the ingredients, we conclude the proof as follows:
\begin{multline*}
  \bigcup_{j=0}^{j^*(t)}I_j(t)=\left(\bigcup_{j=0}^{j^*(t)-1}I_j(t)\right)\cup
  I_{j^*(t)}(t)=\left(\bigcup_{a=t-j^*(t)+1}^{t}I_{t-a}(t)\right)\cup I_{j^*(t)}(t)\\
  \stackrel{(\ref{eq:small_j}),(\ref{eq:jstar})}{=}\bigcup_{a=t-j^*(t)+1}^{t}\left[\binom{a}{2},\,\binom{a}{2}+\binom{\min\{a,t-a\}
        +1}{2}\right]\cup \bigcup_{a=0}^{t-j^*(t)}\left[\binom{a}{2},\,\binom{a}{2}+\binom{\min\{a,t-a\}
        +1}{2}\right]\\
    \stackrel{~(\ref{eq:interval_split})}{=}I(t).\qedhere
\end{multline*}
\end{proof}
To complete the proof of \Cref{thm:shadow_spectrum}, we need the following bound.
\begin{lemma}\label{lem:bound}
  For positive integers $t$, $l$ and $l'$ with $t/2>l\geq l'\geq\left\lceil\frac{t-l}{2}\right\rceil$ and $l>1$,
  \[\binom{l}{2}+\binom{l'}{2}+\binom{t-l-l'+1}{2}+(t-l)\leq\binom{\left\lceil
        t/2\right\rceil}{2}+\binom{\left\lfloor t/2\right\rfloor+1}{2}.\]
\end{lemma}
\begin{proof}
  Note that the assumptions imply $t\geq 5$. Fix $t\geq 5$ and set
  $f(l,l')=\binom{l}{2}+\binom{l'}{2}+\binom{t-l-l'+1}{2}+(t-l)$. If
  $l>l'\geq\left\lceil\frac{t-l}{2}\right\rceil$ then
  \[f(l,l'+1)-f(l,l')=l'-(t-l-l')=2l'-(t-l)\geq 0.\]
  Hence, we can assume $l'=l$, and our task is reduced to showing that
  \[2\binom{l}{2}+\binom{t-2l+1}{2}+(t-l)\leq \binom{\left\lceil
        t/2\right\rceil}{2}+\binom{\left\lfloor t/2\right\rfloor+1}{2}\]
  whenever $t/2>l\geq \left\lceil\frac{t-l}{2}\right\rceil$. Now
  \[f(l+1,l+1)-f(l,l)=2l-(t-2l)-(t-2l-1)-1=6l-2t\geq 0,\]
  hence, we can assume $l=\lceil t/2\rceil-1$, and our task is reduced to verifying
  \[2\binom{\lceil t/2\rceil-1}{2}+\binom{t-2\lceil t/2\rceil+3}{2}+\left\lfloor\frac{t}{2}\right\rfloor+1\leq \binom{\left\lceil
        t/2\right\rceil}{2}+\binom{\left\lfloor t/2\right\rfloor+1}{2}.\]
  For even $t$, this follows from
 \[2\binom{t/2-1}{2}+3+\frac{t}{2}+1=\binom{t/2}{2}+\binom{t/2-1}{2}+5\leq
   \binom{t/2}{2}+\binom{t/2+1}{2},\]
 where we used $t\geq 6$. For odd $t$, we conclude by noticing
 \[2\binom{\frac{t-1}{2}}{2}+1+\frac{t-1}{2}+1=\binom{\frac{t+1}{2}}{2}+\binom{\frac{t-1}{2}}{2}+2\leq 2\binom{\frac{t+1}{2}}{2}.\qedhere\]
\end{proof}
\begin{lemma}\label{lem:superset}
  For integers $k$ and $t$ with $k\geq 2$ and $1\leq t\leq k+1$,
  $\sigma(t,k)\subseteq\{tk-x\,:\,x\in I(t)\}$.
\end{lemma}
\begin{proof}
  The claim is trivial for $t=1$, and we assume $t>1$.  Let $\mathcal{F}$ be a $t$-family of
  $k$-sets. By \Cref{prop:interval_description}, it is sufficient to show that
  $\abs{\Delta \mathcal{F}} \in \{tk-x\,:\,x\in I_{j}(t)\}$ for some $j$. Let
  $\mathcal{A} \subseteq \mathcal{F}$ be a subset of maximum cardinality subject to
  $\abs{\bigcup_{X\in \mathcal{A}} X} \le k+1$. Further, let $l=\abs{\mathcal{A}}$ and
  $\mathcal{B}=\mathcal{F} \setminus \mathcal{A}$, so that $\abs{\mathcal{B}} = t-l$. If $l=1$, then $|\Delta \mathcal{F}|=tk$ is in $\{tk-x\,:\,x\in I(t)\}$ because $0 \in I_{j^*(t)}(t) \subseteq I(t)$. We now assume
  $l>1$. For the shadow sizes we have the bounds
  \begin{align}
    \abs{\Delta \mathcal{A}} &= lk - \binom{l}{2}, \label{eq:A_shadow}\\
    (t-l)k - \binom{t-l}{2} \le \abs{\Delta \mathcal{B}} &\le (t-l)k,\label{eq:B_shadow}\\
    \abs{\Delta \mathcal{A} \cap \Delta \mathcal{B}} &\le t-l.\label{eq:common_shadow}
  \end{align}
  We use induction on $t$ to show the following claim.
  \begin{claim}\label{claim:shadow_bound}
    If $l\geq t/2$, then
    $tk-\binom{l}{2}-\binom{t-l+1}{2}\leq\abs{\Delta\mathcal F}\leq tk-\binom{l}{2}$, and if $l<t/2$,
    then
    $\abs{\Delta \mathcal{F}} \ge tk - \binom{\lceil t/2 \rceil}{2} - \binom{\lfloor t/2 \rfloor
      +1}{2}$.
  \end{claim}
  Before proving this claim, we show how it can be used to deduce the statement of the lemma. For
  $l\geq t-j^*(t)+1$, the claim implies $\abs{\Delta \mathcal{F}} \in \{tk-x\,:\,x\in
  I_{t-l}(t)\}$. For $l\leq t-j^*(t)$, we obtain
  $\abs{\Delta \mathcal{F}} \in \{tk-x\,:\,x\in I_{j^*(t)}(t)\}$ as follows. We have
  $I_{j^*(t)}(t)=[0,f(j^*(t))]$ for the function $f:j\mapsto\binom{t-j}{2}+\binom{j+1}{2}$ which is
  non-increasing for $j\in\{0,1,\dots,\lfloor t/2\rfloor\}$. For $t/2\leq l\leq t-j^*(t)$, 
  \[\abs{\Delta\mathcal F}\geq tk-\binom{l}{2}-\binom{t-l+1}{2}=tk-f(t-l)\geq tk-f(\lfloor
    t/2\rfloor).\]
  For $l<t/2$, $\abs{\Delta\mathcal F}\geq tk-f(\lfloor t/2\rfloor)$ and $\abs{\Delta\mathcal F}\geq
  tk-f(j^*(t))$ follows from $j^*(t)\leq \lfloor t/2\rfloor$.

  It remains to establish \Cref{claim:shadow_bound}. The base case $t=1$ is
  trivial, so we assume $t \ge 2$. Combining~(\ref{eq:A_shadow}),~(\ref{eq:B_shadow})
  and~(\ref{eq:common_shadow}), we obtain
  \[tk-\binom{l}{2} - \binom{t-l}{2} - (t-l)\le\abs{\Delta \mathcal{F}} \le tk - \binom{l}{2}.\] As
  the left-hand side equals $tk-\binom{l}{2}-\binom{t-l+1}{2}$ this concludes the argument for
  $l\geq t/2$. For $l < t/2$, let $\mathcal{C} \subseteq \mathcal{B}$ be a subset of maximum
  cardinality subject to $\abs{\bigcup_{X\in \mathcal{C}} X} \le k+1$. Further, let
  $l'=\abs{\mathcal{C}}$ and $\mathcal{D}:=\mathcal{B} \setminus \mathcal{C}$, so that
  $\abs{\mathcal{D}}=t-l-l'$. The maximality of $\mathcal{A}$ implies $l'\le l$.
  \begin{description}
  \item[Case 1] $l' \ge \frac{t-l}{2}$. Applying the induction hypothesis to $\mathcal{B}$, we obtain
    \[(t-l)k-\binom{l'}{2}-\binom{t-l-l'+1}{2}\leq\abs{\Delta\mathcal B}\leq (t-l)k-\binom{l'}{2}.\]
    Combining this with~(\ref{eq:A_shadow}) and~(\ref{eq:common_shadow}),
    \[\abs{\Delta \mathcal{F}} \ge tk - \binom{l}{2} - \binom{l'}{2} - \binom{t-l-l'+1}{2} -(t-l).\]
  \item[Case 2] $l' < \frac{t-l}{2}$. By induction,
    \[\abs{\Delta \mathcal{B}} \ge (t-l)k - \binom{\left\lceil \frac{t-l}{2}\right \rceil}{2} -
      \binom{\left\lfloor \frac{t-l}{2}\right \rfloor + 1}{2},\]
    and together with~(\ref{eq:A_shadow}) and~(\ref{eq:common_shadow}),
    \[\abs{\Delta \mathcal{F}} \ge tk - \binom{l}{2} - \binom{\left\lceil \frac{t-l}{2}
        \right\rceil}{2}
      - \binom{\left\lfloor \frac{t-l}{2}\right \rfloor + 1}{2} -(t-l).\]
  \end{description}
   In both cases, the claim follows with \Cref{lem:bound}.
\end{proof}
Combining \Cref{lem:subset,lem:superset}, we have
proved~(\ref{eq:shadow_spectrum}). The next lemma provides the final part of \Cref{thm:shadow_spectrum}.
\begin{lemma}\label{lemma:ground_set_small}
  Let $k \ge 2$ and $1\le t \le k+1$ be integers. Then for every $s \in \sigma(t,k)$ there exists a
  set system $\mathcal{F} \subseteq \binom{[k+4]}{k}$ such that $|\mathcal{F}|=t$ and
  $\abs{\Delta \mathcal{F}} = s$.
\end{lemma}

To prove this we need one more lemma which says that for a $(t-1)$-family of $k$-sets on the ground
set $[k+4]$ we can always add a $k$-set that yields $k$ new shadow elements.

\begin{lemma}\label{lemma:append}
  Let $k\geq 2$ and $t$ be integers with $1\leq t\leq k+1$, and assume that
  $\mathcal F\subseteq\binom{[k+4]}{k}$ satisfies $\abs{\mathcal F}=t-1$. Then there exists a set
  $A\in\binom{[k+4]}{k}$ such that $\mathcal F'=\mathcal F\cup\{A\}$ satisfies
  $\abs{\Delta\mathcal F'}=\abs{\Delta\mathcal F}+k$.
\end{lemma}
\begin{proof} We have to show, that for $\mathcal F\subseteq\binom{[k+4]}{k}$ with $\abs{\mathcal
    F}\leq k$, there exists an $A\in\binom{[k+4]}{k}\setminus\nabla\Delta\mathcal F$
  (taking the shade in the ground set $[k+4]$). For $k=2$, there are two elements $a,b\in[k+4]=[6]$,
  which do not appear in any $X\in\mathcal F$, and then $A=\{a,b\}$ does the job. For $k=3$, assume
  for the sake of contradiction, that $\nabla\Delta\mathcal F=\binom{[7]}{3}$. Then
  $E=\binom{[7]}{2}\setminus\Delta\mathcal F$ is the edge set of a triangle-free graph on the vertex
  set $[7]$ with $\abs{E}=21-\abs{\Delta\mathcal F}\geq 12$. By Mantel's theorem, it follows that
  $\Delta\mathcal F=\binom{A}{2}\cup\binom{B}{2}$ for some partition $[7]=A\cup B$ with $\abs{A}=4$
  and $\abs{B}=3$. But then $\abs{F}>3$, which is the required contradiction. 
  For $k\geq 4$, the claim is immediate from
  \[\abs{\nabla\Delta\mathcal F}\leq \abs{\mathcal F}+4\abs{\Delta\mathcal F}\leq k+4k^2<\binom{k+4}{k},\]
  where the last inequality holds because $k\ge 4$. 
\end{proof}

\begin{proof}[Proof of \Cref{lemma:ground_set_small}]
  We proceed by induction on $t$. The base case $t=1$ is trivial, and we assume $t \ge 2$. Let
  $s \in \sigma(t,k)$. By~(\ref{eq:shadow_spectrum}), $s=tk-\binom{a}{2}-\binom{b}{2}-c$ for
  integers $a$, $b$, $c$ satisfying $a\geq b\geq c\geq 0$ and $1\leq a+b\leq t$. If $a+b\leq t-1$
  then~(\ref{eq:shadow_spectrum}) implies that
  $s-k=s=(t-1)k-\binom{a}{2}-\binom{b}{2}-c\in\sigma(t-1,k)$. By induction, there exists
  $\mathcal F'\subseteq\binom{[k+4]}{k}$ with $\abs{\mathcal F'}=t-1$ and
  $\abs{\Delta\mathcal F'}=s-k$, and then, by \Cref{lemma:append}, we can add one more $k$-set
  to obtain $\mathcal F\subseteq\binom{[k+4]}{k}$ with $\abs{\mathcal F}=t$ and
  $\abs{\Delta\mathcal F}=s$. Only the case $a+b=t$ remains. If $c\geq 1$, then $a+b-c\leq t-1\leq k$,
  and \Cref{re:F(abc)} ensures that $\mathcal F_k(a,b,c)\subseteq\binom{[k+2]}{k}$ does the
  job. If $c=0$ then we can assume that $b\geq 1$ because $\binom{1}{2}=\binom{0}{2}$. But then
  $s=tk-\binom{a}{2}-\binom{b}{2}-0=tk-\binom{a}{2}-\binom{b-1}{2}-(b-1)$, and we can take $\mathcal
  F_k(a,b-1,b-1)$.
\end{proof}

Now we put together the ingredients to prove \Cref{thm:shadow_spectrum}.
\begin{proof}[Proof of \Cref{thm:shadow_spectrum}]
  From \Cref{lem:subset,lem:superset}, we
  obtain~(\ref{eq:shadow_spectrum}). The final claim is \Cref{lemma:ground_set_small}.
\end{proof}

\begin{remark}\label{rem:small_ground_set}
  It follows easily from the proofs of \Cref{lemma:append,lemma:ground_set_small}
  that for $t\leq\frac{k+2}{2}$ even the ground set $[k+2]$ is sufficient. For the general statement, the bound
  $k+4$ cannot be improved, because for $k=2$ and $t=3$ the ground set must have at least $6$
  elements to allow a shadow of size $tk=6$.
\end{remark}

\section{Sizes of large maximal antichains}\label{sec:MAC_sizes}
In this section, we establish Theorem~\ref{thm:main_result}(i) by proving the following result which,
in view of Theorem~\ref{thm:shadow_spectrum}, is an equivalent reformulation.
\begin{theorem}\label{thm:large_sizes_reformulation}
  Let $n$ be a positive integer, set $k=\lceil n/2\rceil$ and let $m$ be a positive integer with
  $\binom{n}{k}-k^2 \le m \le \binom{n}{k}$. There exists a maximal antichain of size $m$ in $B_n$ if and only if
  \[\binom{n}{k}-m\in\bigcup_{t=0}^k\left(\sigma(t,k)\cup\sigma(t,n-k)\right).\]
\end{theorem}
In order to obtain a maximal antichain $\mathcal A$ whose size is close to $\binom{n}{k}$, it is
natural to take a small family $\mathcal F\subseteq\binom{[n]}{l+1}$, $l\in\{k,n-k\}$, say with
$\abs{\mathcal F}=t$ and set
$\mathcal A=\mathcal F\cup\binom{[n]}{l}\setminus\Delta\mathcal F$. Then
\[m=\abs{\mathcal F}+\binom{n}{k}-\abs{\Delta\mathcal F}=\binom{n}{k}-(s-t)\]
for some $s\in\sigma(t,l+1)$, and $\binom{n}{k}-m\in\sigma(t,l)$ follows since~(\ref{eq:shadow_spectrum}) implies
\begin{equation}\label{eq:sigma_recursion}
  \sigma(t,l)=\sigma(t,l-1)+t\quad\text{for all }t,l\text{ with }1\leq t\leq l.
\end{equation}

Recall that $S(n)$ denotes the set of sizes of maximal antichains in $B_n$. In the proof of \Cref{thm:large_sizes_reformulation} we will see that for $m\in S(n)$ with
$m\geq\binom{n}{k}-k^2$ there is always an antichain of size $m$ which has this form (except for
some very small values of $n$ which have to be treated separately). We start by proving the easier of the two
directions in \Cref{thm:large_sizes_reformulation}.

\begin{lemma}\label{lem:standard_ac_is_max}
  Let $1\le k \le n$. For every $\mathcal{F} \subseteq \binom{[n]}{k}$ with $\abs{\mathcal{F}}<k$,
  the antichain $\mathcal{A} = \mathcal{F} \cup \binom{[n]}{k-1} \setminus \Delta \mathcal{F}$ is
  maximal.
\end{lemma}

\begin{proof}
  If $\mathcal A$ is not maximal, then there exists $X\in\binom{[n]}{k}\setminus\mathcal F$ with
  $\Delta X\subseteq\Delta\mathcal F$. But $\abs{\Delta X\cap\Delta A}\leq 1$ for every
  $A\in\mathcal F$, hence $\abs{\Delta X\cap\Delta\mathcal F}\leq\abs{\mathcal F}<k=\abs{\Delta X}$,
  a contradiction.
\end{proof}

\begin{lemma}\label{lem:large_sizes_if}
  Let $n$ be a positive integer, set $k=\lceil n/2\rceil$ and let $m$ be a positive integer with
  $\binom{n}{k}-k^2 \le m \le \binom{n}{k}$. If $m=\binom{n}{k}-s$ with
  $s\in\sigma(t,k)\cup\sigma(t,n-k)$ for some $t\in\{0,\dots,k\}$ then $m\in S(n)$.
\end{lemma}
In view of the discussion above, this lemma looks almost trivial, but a bit of work is still needed, because for $s\in\sigma(t,l)$,
Theorem~\ref{thm:shadow_spectrum} guarantees the existence of $\mathcal F\subseteq\binom{[n]}{l+1}$
with $\abs{\mathcal F}=t$ and $\abs{\Delta\mathcal F}=s+t$ only if $n\geq(l+1)+4$.
\begin{proof}(Lemma~\ref{lem:large_sizes_if})
  If $n\geq 10$, then $n\geq (k+1)+4$ and, by Theorem~\ref{thm:shadow_spectrum}, there exists a family
  $\mathcal F\subseteq\binom{[n]}{l+1}$, $l\in\{k,n-k\}$ with $\abs{\mathcal F}=t$ and
  $\abs{\Delta\mathcal F}=t+s$. The antichain
  $\mathcal A=\mathcal F\cup\binom{[n]}{l}\setminus\Delta\mathcal F$ has size $\abs{\mathcal A}=m$ and is maximal by Lemma~\ref{lem:standard_ac_is_max},
  and thus, $m\in S(n)$. We now assume $n\leq 9$. If $m\leq n$ then the antichain
  $\mathcal A=\{\{i\}\,:\,i=1,\dots,m-1\}\cup\{\{m,m+1,\dots,n\}\}$ has size $\abs{\mathcal
    A}=m$. As the antichain $\mathcal A=\binom{[n]}{k}$ has size $\binom{n}{k}$, we can assume
  $n<m<\binom{n}{k}$.  For $n\leq 3$, there is no such $m$, and for $n=4$,
  Theorem~\ref{thm:shadow_spectrum} implies that $m=5$ does not satisfy the assumption of the
  lemma. From now on, we specify a maximal flat antichain
  $\mathcal A\subseteq\binom{[n]}{l}\cup\binom{[n]}{l-1}$ by writing down its $l$-sets. For $n=5$,
  we need the sizes $m\in\{6,7,8\}$. These are obtained by the maximal antichain in
  $\binom{[5]}{3}\cup\binom{[5]}{2}$ with $3$-sets $\{1,2,3\}$ and $\{1,4,5\}$, the maximal antichain in
  $\binom{[5]}{4}\cup\binom{[5]}{3}$ with $4$-set $\{1,2,3,4\}$ and the maximal antichain in
  $\binom{[5]}{3}\cup\binom{[5]}{2}$ with $3$-set $\{1,2,3\}$. For $n\geq 6$, we use graphs similar to
  $G(a,b,c)$, sometimes extended by some additional vertex disjoint edges. More precisely, if $G$ is
  a triangle-free graph on the vertex set $[l+2]$ with $t$ edges and $c$ pairs of adjacent edges,
  then the maximal flat antichain on levels $l$ and $l-1$ whose $l$-sets are the sets
  $[l+2]\setminus e$, $e\in E(G)$, has size $\binom{n}{l-1}-t(l-1)+c$. We will use the graphs 
  with $t$ edges and $c$ pairs of adjacent edges with the following edge sets $E_{tc}$. For brevity, for an edge $\{a,b\}$ we just write $ab$.
  \begin{itemize}
  \item $E_{10}=\{12\}$,
  \item $E_{21}=\{12,13\}$, $E_{20}=\{12,34\}$,
  \item $E_{33}=\{12,13,14\}$, $E_{32}=\{12,23,34\}$ $E_{31}=\{12,13,45\}$, $E_{30}=\{12,34,56\}$
  \item $E_{44}=\{12,13,14,25\}$, $E_{43}=\{12,13,14,56\}$, $E_{42}=\{12,23,34,56\}$,
    $E_{41}=\{12,13,45,67\}$, $E_{40}=\{12,34,56,78\}$
  \item $E_{54}=\{12,13,14,25,67\}$, $E_{53}=\{12,13,14,56,78\}$, $E_{52}=\{12,23,34,56,78\}$,
    $E_{51}=\{12,13,45,67,89\}$
  \item $E_{65}=\{12,13,14,25,56\}$
  \end{itemize}
  \begin{description}
  \item[$n=6$] We need the sizes $m\in\{11,\dots,15, 17\}$. These are obtained by the maximal flat antichains on
    levels 4 and 3, corresponding to the graphs with the edge sets $E_{10}$, $E_{21}$, $E_{20}$, $E_{32}$, $E_{31}$, $E_{30}$.
  \item[$n=7$] We need the sizes $m\in\{19,\dots,32\}$.  The antichains on levels 5 and 4
    corresponding to the graphs with edge sets $E_{10}$, $E_{2c}$, $c\in\{0,1\}$, $E_{3c}$,
    $c\in\{0,\dots,3\}$, $E_{4c}$, $c\in\{1,2,3\}$ and $E_{54}$ have sizes $31$, $28$, $27$, $26$,
    $25$, $24$, $23$, $22$, $21$, $20$ and $19$. The missing sizes $32$, $30$ and $29$ are obtained
    by the antichains on levels 4 and 3 corresponding to the graphs with edge sets $E_{10}$,
    $E_{21}$ and $E_{20}$.
  \item[$n=8$] We need the sizes $m\in\{54,\dots,63, 66\}$. These are obtained by the maximal
    antichains on levels 5 and 4 corresponding to the edge sets $E_{10}$, $E_{21}$, $E_{20}$,
    $E_{3c}$, $c\in\{0,\dots,3\}$, and $E_{4c}$, $c\in\{0,\dots,4\}$.
  \item[$n=9$] We need the sizes $m\in\{101,\dots, 119, 121, 122\}$. The sizes of the maximal
    antichains on levels $6$ and $5$, and on levels 5 and 4, respectively, corresponding to the edge
    sets are indicated in Table~\ref{tab:MAC_sizes}, and this concludes the proof.\qedhere
    \begin{table}[htb]
      \centering
      \caption{The sizes of the maximal flat antichains corresponding to edge sets for $n=9$.}
      \label{tab:MAC_sizes}
      \begin{tabular}{lll} \toprule
        edge sets & levels 6 and 5 & levels 5 and 4\\ \midrule
        $E_{10}$ & $121$ & $122$\\
        $E_{20},\,E_{21}$ & $116$, $117$ & $118$, $119$\\
        $E_{30},\dots,E_{33}$ & $111,\dots,114$ & $114,\dots,117$\\
        $E_{40},\dots,E_{44}$ & $106,\dots,110$ & $110,\dots,114$\\
        $E_{51},\dots,E_{54}$ & $102,\dots,105$ & $107,\dots,110$\\
        $E_{65}$ & $101$ & $107$ \\ \bottomrule
      \end{tabular}
    \end{table}
  \end{description}
\end{proof}

The hard direction in Theorem~\ref{thm:large_sizes_reformulation} is the implication
\begin{equation}\label{eq:contrapositive}
  \binom{n}{k}-m\not\in\bigcup_{t=0}^k\left(\sigma(t,k)\cup\sigma(t,n-k)\right)\implies m\not\in S(n)
\end{equation}
for all $m$ with $\binom{n}{k}-k^2\leq m\leq\binom{n}{k}$. Lemma~\ref{lem:very_large} below allows us to assume $m\geq\binom{n}{k}-\binom{k+1}{2}+2$. In the proof of Lemma~\ref{lem:very_large} we will make use of the following explicit formula for $j^*(t)$.

\begin{lemma}\label{lem:jstar}
	$j^*(t)=\left\lceil\sqrt{2t}-5/2\right\rceil$.
\end{lemma}
\begin{proof}
	By definition, $\binom{j+3}{2}\geq t$, hence $(j+3)(j+2)\geq
	2t$. This implies $j^2+5j+6-2t\geq 0$, and then
	\[j\geq-\frac52+\sqrt{\frac{25}{4}-6+2t}=\frac{\sqrt{8t+1}-5}{2}.\]
	From this, we obtain
	\[j^*(t)=\left\lceil\frac{\sqrt{8t+1}-5}{2}\right\rceil\geq\frac{\sqrt{8t}-5}{2}=\sqrt{2t}-\frac52.\]
	If $j^*(t)>\left\lceil\sqrt{2t}-5/2\right\rceil$ then
	\[\sqrt{2t}-\frac52\leq l<\frac{\sqrt{8t+1}-5}{2}\]
	for some integer $l$. But this implies $8t\leq(2l+5)^2<8t+1$, which is impossible, because the
	number in the middle is an odd integer.
\end{proof}

\begin{lemma}\label{lem:very_large}
  Let $n$ be a positive integer, set $k=\lceil n/2\rceil$ and let $m$ be a positive integer with
  $\binom{n}{k}-k^2 \le m \le \binom{n}{k}$. If $\binom{n}{k}-m\not\in\bigcup_{t=0}^k\sigma(t,k)$ then $m\geq\binom{n}{k}-\binom{k+1}{2}+2$.
\end{lemma}
\begin{proof}
  The claim follows from the inclusion
  \begin{equation}\label{eq:interval}
    \left[\binom{k+1}{2}-1,\ k^2\right]\subseteq\bigcup_{t=0}^k\sigma(t,k).
  \end{equation}
  For $1\leq k\leq 4$, we verify this by hand: $\{0,1\}\subseteq\sigma(0,1)\cup\sigma(1,1)$,
  $[2,4]\subseteq\sigma(1,2)\cup\sigma(2,2)$, $[5,9]\subseteq\sigma(2,3)\cup\sigma(3,3)$, and
  $[9,16]\subseteq\sigma(3,4)\cup\sigma(4,4)$. In order to prove~(\ref{eq:interval}) for $k\geq 5$,
  we use $\sigma(t,k)\supseteq \{tk-x : x \in I_{j^*(t)}(t)\}=\left[tk-f(t),\,tk\right]$, where
  \[f(t)=\binom{t-\lceil\sqrt{2t}-5/2\rceil}{2}+\binom{\lceil\sqrt{2t}-5/2\rceil+1}{2},\]
  and we take the union only over $t\geq\left\lceil(k+1)/2\right\rceil$. In other words, we verify the
  following inequalities:
  \begin{enumerate}[(i)]
  \item For $t=\left\lceil\frac{k+1}{2}\right\rceil$, $tk-f(t)\leq \binom{k+1}{2}-1$, and
  \item for $\left\lceil\frac{k+1}{2}\right\rceil<t\leq k$, $tk-f(t)\leq(t-1)k+1$.
  \end{enumerate}
  Rearranging these inequalities, we want to verify that
  \[f\left(\left\lceil\frac{k+1}{2}\right\rceil\right)\geq
    \begin{cases}
      1 &\text{if $k$ is odd},\\
      \frac{k}{2}+1&\text{if $k$ is even}.
    \end{cases}\]
  and $f(t)\geq k-1$ for $\left\lceil\frac{k+1}{2}\right\rceil<t\leq k$. For
  $5\leq k\leq 26$, this follows from Table~\ref{tab:f} and the fact that $f$ is increasing.
  \begin{table}[htb]
    \caption{The values $f(t)$ for $3\leq t\leq 10$.}
    \label{tab:f}
    \centering
    \begin{tabular}{crrrrrrrr} \toprule
      $t$ &3&4&5&6&7&8&9&10 \\
      $f(t)$ &3&4&7 & 11 & 13 & 18 & 24 & 31 \\ \bottomrule
    \end{tabular}
  \end{table}
  For $k\geq 27$ we show that $f(\lceil k/2\rceil)\geq k$, and this is sufficient by the
  monotonicity of $f$. We start with the bound
  \[f(\lceil k/2\rceil)\geq\binom{\lceil k/2\rceil-\lceil\sqrt{2\lceil
        k/2\rceil}-5/2\rceil}{2}\geq\frac12\left(\frac{k}{2}-\sqrt k\right)\left(\frac{k}{2}-\sqrt
      k-1\right).\]
  Simplifying the inequality
  $\frac12\left(\frac{k}{2}-\sqrt k\right)\left(\frac{k}{2}-\sqrt k-1\right)\geq k$ we find that it
  is equivalent to
  \[0\leq k\sqrt{k}-4k-6\sqrt{k}+4=\sqrt
    k\left(\sqrt{k}-2-\sqrt{10}\right)\left(\sqrt{k}-2+\sqrt{10}\right)+4,\]
  which is true for every $k\geq\left\lceil(2+\sqrt{10})^2\right\rceil=27$.
\end{proof}

From Lemma~\ref{lem:very_large} it follows that in order to complete the proof of
Theorem~\ref{thm:large_sizes_reformulation} by establishing~(\ref{eq:contrapositive}), we need to
prove the following lemma.
\begin{lemma}\label{lem:very_large_sizes_only_if}
  Let $n$ be a positive integer, set $k=\lceil n/2\rceil$ and let $m$ be a positive integer with
  $\binom{n}{k}-\binom{k+1}{2}+2 \le m \le \binom{n}{k}$. If $m\in S(n)$ then $\binom{n}{k}-m\in\bigcup_{t=0}^k\left(\sigma(t,k)\cup\sigma(t,n-k)\right)$.
\end{lemma}
Before proving Lemma~\ref{lem:very_large_sizes_only_if}, we briefly summarize the proof of
Theorem~\ref{thm:large_sizes_reformulation}, assuming the lemmas that have been stated so far.
\begin{proof}[Proof of Theorem~\ref{thm:large_sizes_reformulation}]
  The ``if''-direction of the claimed equivalence is Lemma~\ref{lem:large_sizes_if}. For the
  ``only-if''-direction, we can assume, by Lemma~\ref{lem:very_large}, that
  $m\geq\binom{n}{k}-\binom{k+1}{2}+2$, and then Lemma~\ref{lem:very_large_sizes_only_if} concludes
  the argument.
\end{proof}
It remains to prove Lemma~\ref{lem:very_large_sizes_only_if}, and we obtain this as a consequence of the
following result.
\begin{lemma}\label{lem:equivalent_to_large_sizes}
  Let $n$ be a positive integer, set $k=\lceil n/2\rceil$ and let $m$ be an integer with $m\in S(n)$ and
  $\binom{n}{k}-\binom{k+1}{2}+2\leq m\leq\binom{n}{k}$. Then there exists a family $\mathcal
  F\subseteq\binom{[n]}{l+1}$ with $l\in\{k,n-k\}$ and $\abs{\mathcal F}\leq k+1$ such that
  $m=\binom{n}{k}+\abs{\mathcal F}-\abs{\Delta\mathcal F}$.
\end{lemma}
Before proving Lemma~\ref{lem:equivalent_to_large_sizes} we explain how
Lemma~\ref{lem:very_large_sizes_only_if} follows from it.
\begin{proof}[Proof of Lemma~\ref{lem:very_large_sizes_only_if}]
   It follows from Lemma~\ref{lem:equivalent_to_large_sizes}, that $m\in S(n)$ implies the existence of a family
   $\mathcal F\subseteq\binom{[n]}{l+1}$, $l\in\{k,n-k\}$ with $\abs{\mathcal F}=t\leq k$ and
   \[\binom{n}{k}-m=\abs{\Delta\mathcal F}-\abs{\mathcal
       F}\in\sigma(t,l+1)-t\stackrel{(\ref{eq:sigma_recursion})}{=}\sigma(t,l).\qedhere\]
\end{proof}

The proof of Lemma~\ref{lem:equivalent_to_large_sizes} is a bit involved and uses some auxiliary
lemmas. We start with an antichain $\mathcal A$ of the required large size, and first establish that
this forces $\mathcal A$ to be contained in three central levels of $B_n$
(Lemmas~\ref{lem:A_is_in_the_middle_1} and~\ref{lem:A_is_in_the_middle_2}). In
Lemma~\ref{lem:few_sets_out_of_big_level}, we show that $\mathcal A$ is very close to being a
complete level. Then we need a lemma which says that certain sums of shadow sizes are again shadow
sizes (Lemma~\ref{lem:two_levels_suffice}) to conclude that we can assume that $\mathcal A$ is flat,
that is, it is contained in two consecutive levels (Lemma~\ref{lem:two_levels}).

Let us recall a well known inequality, the \emph{Normalized Matching Property (NMP)}.
If $\mathcal{F}\subseteq\binom{[n]}{k+1}$ for some $k\in [n]$, then every $F\in\mathcal{F}$ contains exactly $k+1$ sets from
$\Delta\mathcal{F}$ and every $X\in\Delta\mathcal{F}$ is contained in at most $n-k$ sets from $\mathcal{F}$.
This implies the NMP, $(n-k)|\Delta\mathcal{F}|\ge (k+1)|\mathcal{F}|$.
As a consequence, $|\Delta\mathcal{F}|\ge |\mathcal{F}|$ whenever $k\geq\frac{n-1}{2}$ with strict inequality for $k>\frac{n-1}2$.
Similarly, one has $|\nabla\mathcal{F}|\ge |\mathcal{F}|$ for $\mathcal{F}\subseteq\binom{[n]}k$ and $k<\frac n2$.
The latter inequalities were the essential tool in Sperner's original proof of his famous theorem.

We will also need the following strengthening for $k=\lceil n/2\rceil$.

\begin{lemma}\label{lem:Delta_minus_size}
Let $n$ be a positive integer and $k=\lceil n/2\rceil$.
\begin{enumerate}[(i)]
  \item If $\mathcal{F}\subseteq\binom{[n]}{k+1}$ with $|\mathcal{F}|\ge k$, then $|\Delta\mathcal{F}|-|\mathcal{F}|\ge \binom{k+1}2-1$.
  \item If $\mathcal{F}\subseteq\binom{[n]}{k+2}$ is nonempty, then $|\Delta\mathcal{F}|-|\mathcal{F}|\ge k+1$.
\end{enumerate}
\end{lemma}

\begin{proof}
\emph{(i)}\/
Let $\mathcal{F}\subseteq\binom{[n]}{k+1}$ with $|\mathcal{F}|\ge k$.
If $|\mathcal{F}|=k$, then, by the Kruskal-Katona Theorem, we have $|\Delta\mathcal{F}|-|\mathcal{F}|\ge k(k+1)-\binom k2-k=\binom{k+1}2$.
If $|\mathcal{F}|=k+1$, then we have $|\Delta\mathcal{F}|-|\mathcal{F}|\ge (k+1)^2-\binom{k+1}2-(k+1)=\binom{k+1}2$.

Assume that $|\mathcal{F}|\ge k+2$.
Then $|\mathcal{F}|=\binom x{k+1}$ for some real number $x\ge k+2$ and,
by Lovasz'~\cite{Lovasz1979} continuous version of the Kruskal-Katona Theorem, $|\Delta\mathcal{F}|\ge\binom xk$.
Hence,
\[ |\Delta\mathcal{F}|-|\mathcal{F}| ~\ge~ \binom xk - \binom x{k+1} ~=~ \frac{2k+1-x}{(k+1)!}\,\prod_{j=0}^{k-1} (x-j) ~=:~ P_k(x) . \]
It is sufficient to show that $P_k(x)\ge \binom{k+1}2-1$ for $k+2\le x\le n$.
The two largest zeros of the polynomial $P_k(x)$ are $k-1$ and $2k+1$ and $P_k(x)>0$ for $k-1<x<2k+1$.
As $n\in\{2k-1,2k\}$, the minimum of $P_k$ on the interval $[k+2,n]$ is attained at one of the endpoints.
We have $P_k(k+2)=\binom{k+2}k-(k+2)=\binom{k+1}2-1$.
Observe that
\[ P_k(2k-1)~=P_k(2k)~=~\frac 1{(k+1)!}\,\prod_{j=k+1}^{2k}j . \]
It follows that the unique maximum of $P_k$ on $[k-1,2k+1]$ is attained between $2k-1$ and $2k$.
Consequently, $P_k(x)$ is increasing on $[k-1,2k-1]$ and attains its minimum on $[k+1,n]$ at $x=k+1$.
This implies the claim.

\emph{(ii)}\/
Here we consider $P_{k+1}$ on the interval $[k+2,n]$. The two largest zeros of $P_{k+1}$ are $k$ and $2k+3$.
Moreover, $P_{k+1}(2k+1)=P_{k+1}(2k+2)$ and $n<2k+1$. Therefore, $P_{k+1}$ is increasing on $[k+2,n]$, and
its minimum on $[k+2,n]$ is $P_{k+1}(k+2)=k+1$.
\end{proof}

\begin{lemma}\label{lem:A_is_in_the_middle_1}
  Let $n$ be a positive integer, set $k=\lceil n/2\rceil$, and let $\mathcal A$ be an antichain in $B_n$ with size
  $\abs{\mathcal A}\geq\binom{n}{k}-\binom{k+2}{2}+2$.
  \begin{enumerate}[(i)]
  \item If $n=2k$ then $\mathcal A\subseteq\binom{[n]}{k-1}\cup\binom{[n]}{k}\cup\binom{[n]}{k+1}$.
  \item If $n=2k-1$ then
    $\mathcal A\subseteq\binom{[n]}{k-2}\cup\binom{[n]}{k-1}\cup\binom{[n]}{k}\cup\binom{[n]}{k+1}$.
  \end{enumerate}
\end{lemma}

\begin{proof}
  We show that all the sets in $\mathcal A$ have size at most $k+1$. As the family consisting
  of the complements of the members of $\mathcal A$ is also an antichain this implies that
  all members of $\mathcal A$ have size at least $k-1$ if $n$ is even, and at least $k-2$ if $n$ is
  odd, and this proves the statement.

  For the sake of contradiction, assume that $\mathcal A$ contains a set of size at least $k+2$.
  As in Sperner's proof of his theorem, we can replace $\mathcal A\cap\binom{[n]}{i}$ by its shade
  if $i=\min\{j:\mathcal{A}\cap\binom{[n]}j\ne\emptyset\}<k$ and by its shadow if
  $i=\max\{j:\mathcal{A}\cap\binom{[n]}j\ne\emptyset\}>k+2$. Repeating this
  yields an antichain $\mathcal A'\subseteq\binom{[n]}{k}\cup\binom{[n]}{k+1}\cup\binom{[n]}{k+2}$
  with $\abs{\mathcal A'}\geq\abs{\mathcal A}$.
  Let $\mathcal B=\mathcal A'\cap\binom{[n]}{k+2}$.
  By Lemma \ref{lem:Delta_minus_size}\,(ii), $\mathcal{A}''=(\mathcal{A}'\setminus\mathcal{B})\cup\Delta\mathcal{B}$ is an antichain
  with $\abs{\mathcal A''}\geq\abs{\mathcal A'}+k+1$ and $|\mathcal A''\cap\binom{[n]}{k+1}|\ge k+2$.
  Let $\mathcal C=\mathcal A''\cap\binom{[n]}{k+1}$.
  Finally, $\mathcal A'''=(\mathcal{A}''\setminus\mathcal{C})\cup\Delta\mathcal{C}$ is an antichain with
  $\abs{\mathcal A'''}=\abs{\mathcal A''}+|\Delta\mathcal C|-|\mathcal C|\le \binom nk$.
  By Lemma \ref{lem:Delta_minus_size}\,(i), $|\Delta\mathcal C|-|\mathcal C|\ge\binom{k+1}2-1$. This implies
  $\abs{\mathcal A''}\le \binom nk-\binom{k+1}2+1$ and $\abs{\mathcal A'}\le \binom nk-\binom{k+2}2+1$,
  a contradiction.
\end{proof}

\begin{lemma}\label{lem:A_is_in_the_middle_2}
  Let $n=2k-1$, and let $\mathcal A$ be an antichain in $B_n$ with size
  $\abs{\mathcal A}\geq\binom{n}{k}-\binom{k+1}{2}+2$. Then
  $\mathcal A\subseteq\binom{[n]}{k-2}\cup\binom{[n]}{k-1}\cup\binom{[n]}{k}$ or
  $\mathcal A\subseteq\binom{[n]}{k-1}\cup\binom{[n]}{k}\cup\binom{[n]}{k+1}$.
\end{lemma}

\begin{proof}
  By Lemma~\ref{lem:A_is_in_the_middle_1} we know that
  $\mathcal A\subseteq\binom{[n]}{k-2}\cup\binom{[n]}{k-1}\cup\binom{[n]}{k}\cup\binom{[n]}{k+1}$.
  Assume that $\mathcal A$ contains at least one $(k+1)$-set and at least one $(k-2)$-set.
  For $i=1,2,3,4$, let $\mathcal A_i=\mathcal A \cap \binom{[n]}{k+2-i}$.
  By the Kruskal-Katona Theorem, without loss of generality, we can assume that $\mathcal A$ is left-compressed,
  i.e., that $\mathcal A_1$, $\mathcal A'_2=\Delta\mathcal A_1\cup\mathcal A_2$, $\mathcal A'_3=\Delta\mathcal A'_2\cup\mathcal A_3$,
  $\mathcal A'_4=\Delta\mathcal A'_3\cup\mathcal A_4$ are initial segments of
  $\binom{[n]}{k+1}$, $\binom{[n]}{k}$, $\binom{[n]}{k-1}$, $\binom{[n]}{k-2}$, respectively, with respect to squashed order.
  Partition $\mathcal A$ into $\mathcal A^-=\{A\in\mathcal A\,:\,n\notin A\}$ and
  $\mathcal A^+=\{A\in\mathcal A\,:\,n\in A\}$.
  As $\mathcal A$ is left-compressed, we have $\mathcal A^-\subseteq\binom{[n]}{k+1}\cup\binom{[n]}{k}$ or
  $\mathcal A^+\subseteq\binom{[n]}{k-1}\cup\binom{[n]}{k-2}$.
  Assume that $\mathcal A^+\subseteq\binom{[n]}{k-1}\cup\binom{[n]}{k-2}$, the other case is analogous.
  Then $|\mathcal A^+|\le \binom{n-1}{k-2}$ and, by Lemma \ref{lem:A_is_in_the_middle_1}\,(i),
  $\abs{\mathcal A^-}\leq\binom{n-1}{k-1}-\binom{k+1}{2}+1$.
  Hence, $\abs{\mathcal A}=\abs{\mathcal A^-}+\abs{\mathcal A^+}\leq\binom{n}{k}-\binom{k+1}{2}+1$.
\end{proof}

At this point, without loss of generality, we can assume that $\mathcal A$ lives on the three levels
$k-1$, $k$ and $k+1$. Next we show that $\mathcal A$ must be almost a complete largest level.

\begin{lemma}\label{lem:few_sets_out_of_big_level}
  Let $n$ be a positive integer, set $k=\lceil n/2\rceil$, and let
  $\mathcal A\subseteq\binom{[n]}{k-1}\cup\binom{[n]}{k}\cup\binom{[n]}{k+1}$ be an antichain with size
  $\abs{\mathcal A}\geq\binom{n}{k}-\binom{k+1}{2}+2$.
  \begin{enumerate}[(i)]
  \item If $n=2k$, then
    $\abs{\mathcal A\cap\binom{[n]}{k-1}}+\abs{\mathcal A\cap\binom{[n]}{k+1}}\leq k+1$.
  \item If $n=2k-1$, then
    $\abs{\mathcal A\cap\binom{[n]}{k-1}}+\abs{\mathcal A\cap\binom{[n]}{k+1}}\leq k+1$ or
    $\mathcal A\subseteq\binom{[n]}{k-1}\cup\binom{[n]}{k}$ and
    $\abs{\mathcal A\cap\binom{[n]}{k}}\leq k+1$.
  \end{enumerate}
\end{lemma}

\begin{proof}
\emph{(i)}\/
  Set $\mathcal A_1=\mathcal A\cap\binom{[n]}{k+1}$, $\mathcal A_2=\mathcal A\cap\binom{[n]}{k-1}$,
  $t_1=\abs{\mathcal A_1}$ and $t_2=\abs{\mathcal A_2}$. If $t_1\geq k+2$, then
  $\abs{\Delta\mathcal A_1}-|\mathcal A_1|\geq \binom{k+1}{2}-1$ by Lemma \ref{lem:Delta_minus_size}\,(i), hence
  \[\abs{\mathcal A}
  ~\le~\binom{n}{k}+\left(\abs{\mathcal A_1}-\abs{\Delta\mathcal A_1}\right)+\left(\abs{\mathcal A_2}-\abs{\nabla\mathcal A_2}\right)
  ~\le~\binom{n}{k}+\left(\abs{\mathcal A_1}-\abs{\Delta\mathcal A_1}\right)
  ~\le~\binom{n}{k}-\binom{k+1}{2}+1.\]
  Consequently, $t_1\leq k+1$ and similarly, $t_2\leq k+1$.
  Now, by Kruskal-Katona, we have $|\Delta\mathcal A_1|\ge (k+1)t_1-\binom{t_1}2$ and $|\nabla\mathcal A_2|\ge (k+1)t_2-\binom{t_2}2$.
  This implies
  \[  \abs{\mathcal A}~\le~\binom{n}{k}-k(t_1+t_2)+\binom{t_1}2+\binom{t_2}2 . \]
  For fixed $t=t_1+t_2$ with $k+2\le t\le 2(k+1)$, the expression on right-hand side of the above inequality
  attains its maximum if $\max\{t_1,t_2\}=k+1$ and $\min\{t_1,t_2\}=t-k-1$.
  Hence, for $k+2\le t\le 2(k+1)$ we have
  \[ \abs{\mathcal A} ~\le~ \binom{n}{k}-kt+\binom{k+1}2+\binom{t-k-1}2
     ~=~ \binom{n}{k}-\binom{k+1}2-\frac 12 (t-k-1)(3k+2-t)
     ~<~ \binom{n}{k}-\binom{k+1}2,\]
  and this concludes the proof.

  \emph{(ii)}\/
  Set $\mathcal A_i=\mathcal A\cap\binom{[n]}{k+2-i}$ and $t_i=\abs{\mathcal A_i}$ for $i=1,2,3$.
  Like in the proof of Lemma \ref{lem:A_is_in_the_middle_2}, without loss of generality, we assume that $\mathcal A$ is left-compressed,
  i.e., $\mathcal A_1$, $\mathcal A'_2=\Delta\mathcal A_1\cup\mathcal A_2$, $\mathcal A'_3=\Delta\mathcal A'_2\cup\mathcal A_3$
  are initial segments of $\binom{[n]}{k+1}$, $\binom{[n]}{k}$, $\binom{[n]}{k-1}$, respectively, with respect to squashed order.
  Furthermore, let $\mathcal A^-=\{A\in\mathcal A\,:\,n\notin A\}$ and $\mathcal A^+=\{A\in\mathcal A\,:\,n\in A\}$.

  \emph{Case 1:}\/
  Assume that $t_1=0$. We have to show that $\min\{t_2,t_3\}\le k+1$. Assume for a contradiction that $t_2\ge k+2$ and $t_3\ge k+2$.
  We have $t_2\le\binom{n-1}k$ or $t_3\le\binom{n-1}k$ because $|\mathcal A_2|>\binom{n-1}k$ implies
  $|\mathcal A_3| \le \binom n{k-1}-|\Delta\mathcal A_2| < \binom n{k-1}-\binom n{k-1} = \binom {n-1}{k-1} = \binom{n-1}{k-1}$.
  Without loss of generality, we assume that $t_2\le\binom{n-1}k$. Then $\mathcal A_2\subseteq\mathcal A^-$,
  and by Lemma \ref{lem:Delta_minus_size}\,(i), we have $|\Delta\mathcal A_2|-|\mathcal A_2|\ge\binom{k+1}2-1$.
  This implies $|\mathcal A| \le \binom n{k-1}-|\Delta\mathcal A_2|+|\mathcal A_2| \le \binom nk-\binom{k+1}2+1$, a contradiction.

  \emph{Case 2:}\/
  Assume that $t_1\ge 1$. For a contradiction, assume further that $t=t_1+t_3\ge k+2$.
  If $t_1\ge k$, then by Lemma \ref{lem:Delta_minus_size}\,(i) and $|\nabla\mathcal A_3|\ge |\mathcal A_3|$, we have
  $$ |\mathcal A| ~\le~ \binom nk-(|\Delta\mathcal A_1|-|\mathcal A_1|)-(|\nabla\mathcal A_3|-|\mathcal A_3|)
  ~\le~ \binom nk-\binom{k+1}2+1, $$
  a contradiction.
  Therefore, $t_1<k$ which implies
  \begin{equation}\label{eq:small_t1}
  \textstyle |\Delta\mathcal A_1|-|\mathcal A_1|~\ge~kt_1-\binom{t_1}2~=~k+(t_1-1)(k-\frac{t_1}2)~\ge~k .
  \end{equation}

  \emph{Case 2.1:}\/
  Assume that $\mathcal A^-\subseteq\mathcal A_1\cup\mathcal A_2$. Then $\mathcal A_3\subseteq\mathcal A^+$.
  If $t_3\ge k$, then by Lemma \ref{lem:Delta_minus_size}\,(i), we have
  $|\nabla\mathcal A_3|-|\mathcal A_3|\ge\binom k2-1$ which, together with (\ref{eq:small_t1}),
  yields $|\mathcal A|\le\binom nk-k-(\binom k2-1)=\binom nk-\binom{k+1}2+1$, a contradiction.
  Hence, $t_3<k$, and a contradiction follows by
  \begin{eqnarray*}
  |\mathcal A| & \le & \textstyle \binom nk-kt_1+\binom{t_1}2-(k-1)t_3+\binom{t_3}2 ~=~ \binom nk-k(t_1+t_3)+\binom{t_1}2+\binom{t_3+1}2\\[.5ex]
   & \le & \textstyle \binom nk-tk+\binom{t-k+1}2+\binom k2 ~=~ \binom nk-\binom{k+1}2-\frac 12(3k-1-t)(t-k)~<~\binom nk-\binom{k+1}2,
  \end{eqnarray*}
where the last equality comes from substituting $\binom{k+1}{2}+\frac12k(2t-k-1)$ for $tk$, and then
rearranging the terms.
  \emph{Case 2.2:}\/
  Assume that $\mathcal A^-\cap\mathcal A_3\ne\emptyset$.
  Then $\mathcal A^+\subseteq \mathcal A_3$, and hence, $|\mathcal A^+|\le\binom{n-1}{k-2}$.
  On the other hand, by \Cref{lem:A_is_in_the_middle_1}\,(i),
  $|\mathcal A^-|\le\binom{n-1}{k-1}-\binom {k+1}2+1$.
  Finally, we obtain $|\mathcal A|=|\mathcal A^-|+|\mathcal A^+|\le \binom nk-\binom{k+1}2+1$, a contradiction.
\end{proof}

In the final step we want to replace a maximal antichain $\mathcal A$ on three consecutive levels by
another maximal antichain on two consecutive levels which has the same size as $\mathcal A$. To
prove that this is always possible we will use the fact that sums of certain shadow sizes are themselves shadow
sizes. Specifically, we need the following result.
\begin{lemma}\label{lem:two_levels_suffice}\hfill
  \begin{enumerate}[(i)]
  \item If $t\leq k$ then
    $\sigma(1,k)+\sigma(t,k-1)\subseteq\sigma(1+t,k)\cup\sigma(1+t,k-1)$.
  \item If $t_1\geq 2$ and $t_1+t_2\leq k+1$ then
    $\sigma(t_1,k)+\sigma(t_2,k-1)\subseteq\sigma(t_1+t_2,k)$.
  \end{enumerate}
\end{lemma}
In the proof of Lemma~\ref{lem:two_levels_suffice} we will use graphs obtained by gluing $G(a,b,c)$
and $G(a',b',c')$, as defined in Section~\ref{sec:shadow_spectrum}, in the following way.
\begin{definition}
  Let $a,b,c,a',b',c'$ be integers with $a\ge b \ge c\ge 0, a'\ge b' \ge c' \ge 0$, and $a,a'\ge
  1$. Set $G=G(a,b,c)$ and $G'=G(a',b',c')$. Moreover, assume that $G$ has at least one pendant
  vertex if $b'=0$ and at least two pendant vertices if $b'\geq 1$.
  \begin{itemize}
  \item For $b'=0$, let $G \oplus G'$ denote a graph that is obtained from $G$
    and $G'$ by identifying one pendant vertex of $G$ with the center of the $a'$-star $G'$.
  \item For $b'\ge 1$, let $G \oplus G'$ denote a graph that is obtained from $G$ and $G'$ by
    identifying two pendant vertices of $G$ with the centers of the two stars forming $G'$.
  \end{itemize}
\end{definition}
The assumption on the number of pendant vertices in $G$ is satisfied whenever any of the following
statements is true:
\begin{multicols}{3}
\begin{enumerate}[(i)]
\item $a+b-2c\geq 2$, or
\item $b'=0$, $a+b-2c\geq 1$, or
\item $b=1$, $a-2c\geq 0$.
\end{enumerate}
\end{multicols}
The construction for $b'\geq 1$ and $a+b-2c\geq 2$ is illustrated in Figure~\ref{fig:gluing}. The
given description does not characterize the graph $G\oplus G'$ up to isomorphism, as we do not
specify which of the pendant vertices of $G$ are used for gluing $G'$. For our purpose this is not
important as we will only use the number of edges and the number of pairs of adjacent edges. The
graph $G(a,b,c) \oplus G(a',b',c')$ has $a+a'+b'-c'+1$ vertices if $b=0$ and $a+b-c+a'+b'-c'+2$
vertices if $b\geq 1$. The number of edges is always $a+b+a'+b'$ and there are
\[\binom{a}{2}+\binom{b}{2}+c+\binom{a'}{2}+\binom{b'}{2}+c' + a' + b'\]
pairs of adjacent edges.
\begin{definition}\label{def:F(a,b,c,a',b',c')}
  For integers $k$ and $a,b,c,a',b',c'$ with $a\ge b \ge c\ge 0$ and $a'\ge b' \ge c'\ge 0$, let
  \[\mathcal{G}_k(a,b,c,a',b',c') = \left\{[k+2]\setminus e\,:\, e \textnormal{ is an edge of }
    G(a,b,c)\oplus G(a',b',c')\right\}.\]
\end{definition}
\begin{figure}[htb]
      \centering
    \begin{tikzpicture}[scale=1.5,every node/.style={draw,circle,fill,outer sep=1pt,inner
        sep=1pt}]
           \node (a) at (0,1.5) {};
      \node (b) at (3,1.5) {};
      \node (v1) at (-.5,0) {};
      \node[draw=none,fill=none] at (0,0) {$\ldots$};
      \node (v2) at (.5,0) {};
      \node (v3) at (1,0) {};
      \node[draw=none,fill=none] at (1.5,0) {$\ldots$};
      \node (v4) at (2,0) {};
      \node (v5) at (2.5,0) {};
      \node[draw=none,fill=none] at (3,0) {$\ldots$};
      \node (v6) at (3.5,0) {};
      \draw[thick] (a) -- (v1);
      \draw[thick] (a) -- (v2);
      \draw[thick] (a) -- (v3);
      \draw[thick] (a) -- (v4);
      \draw[thick] (b) -- (v3);
      \draw[thick] (b) -- (v4);
      \draw[thick] (b) -- (v5);
      \draw[thick] (b) -- (v6);
      \draw [thick,decoration={brace,mirror,raise=0.1cm},decorate] (.95,0) -- (2.05,0)
      node [fill=none,draw=none,pos=0.5,anchor=north,yshift=-0.2cm] {$c$};
      \draw[thick,domain=-120:-25] plot ({0.4*cos(\x)},{1.5+0.4*sin(\x)});
      \node[fill=none,draw=none] at (-.3,1.2) {$a$};
      \draw[thick,domain=-150:-60] plot ({3+0.4*cos(\x)},{1.5+0.4*sin(\x)});
      \node[fill=none,draw=none] at (3.3,1.2) {$b$};
      \node (w1) at (-.5,-1.5) {};
      \node[draw=none,fill=none] at (0,-1.5) {$\ldots$};
      \node (w2) at (.5,-1.5) {};
      \node (w3) at (1,-1.5) {};
      \node[draw=none,fill=none] at (1.5,-1.5) {$\ldots$};
      \node (w4) at (2,-1.5) {};
      \node (w5) at (2.5,-1.5) {};
      \node[draw=none,fill=none] at (3,-1.5) {$\ldots$};
      \node (w6) at (3.5,-1.5) {};
      \draw[thick] (v1) -- (w1);
      \draw[thick] (v1) -- (w2);
      \draw[thick] (v1) -- (w3);
      \draw[thick] (v1) -- (w4);
      \draw[thick] (v5) -- (w3);
      \draw[thick] (v5) -- (w4);
      \draw[thick] (v5) -- (w5);
      \draw[thick] (v5) -- (w6);
      \draw [thick,decoration={brace,mirror,raise=0.1cm},decorate] (.95,-1.5) -- (2.05,-1.5)
      node [fill=none,draw=none,pos=0.5,anchor=north,yshift=-0.2cm] {$c'$};
      \draw[thick,domain=-107:-20] plot ({-.5+.4*cos(\x)},{0.4*sin(\x)});
      \node[fill=none,draw=none] at (-.7,-.35) {$a'$};
      \draw[thick,domain=-145:-40] plot ({2.5+.4*cos(\x)},{0.4*sin(\x)});
      \node[fill=none,draw=none] at (2.95,-.25) {$b'$};
    \end{tikzpicture}
    \caption{A graph $G(a,b,c)\oplus G(a',b',c')$.}\label{fig:gluing}
\end{figure}
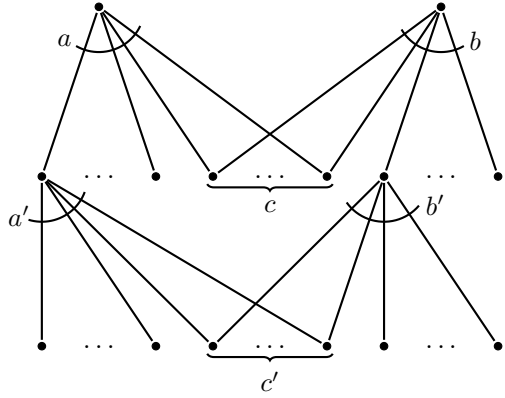
\begin{remark}\label{re:G(a,b,c,a',b',c')}
  If $b=0$ and $k\geq a+a'+b'-c'-1$ or $b\geq 1$ and $k\geq a+b-c+a'+b'-c'$ then the vertex set of
  $G(a,b,c)\oplus G(a',b',c')$ is contained in $[k+2]$, and consequently,
  $\mathcal G_k(a,b,c,a',b',c')$ is a family of $k$-sets with
  $\abs{\mathcal{G}_k(a,b,c,a',b',c')} = a+b+a'+b'$. Moreover, in this situation, the size of its shadow
  equals its size multiplied by $k$ reduced by the number of pairs of sets that have a shadow
  element in common. As the latter equals the number of pairs of adjacent edges of
  $G(a,b,c)\oplus G(a',b',c')$, we obtain
  \[\abs{\Delta \mathcal{G}_k(a,b,c,a',b',c')} = (a+b+a'+b')k - \binom{a}{2} - \binom{b}{2} - c
    -\binom{a'}{2}-\binom{b'}{2}-c'-a'-b'.\]
  We also note that
  $\abs{\Delta \mathcal{G}_k(a,b,c,a',b',c')}=\abs{\Delta \mathcal{F}_k(a,b,c)} + \abs{\Delta
    \mathcal{F}_{k-1}(a',b',c')}$.
\end{remark}
\begin{proof}[Proof of Lemma~\ref{lem:two_levels_suffice}]
  For (i), assume $t\leq k$ and let $s \in \sigma(t,k-1)$. By~(\ref{eq:shadow_spectrum}), there are
  integers $a \ge b \ge c \ge 0$ with $1\le a+b \le t$ and $s =
  t(k-1)-\binom{a}{2}-\binom{b}{2}-c$. If $c\geq 1$, then
  \[s+k = (t+1)(k-1) - \binom{a}{2} - \binom{b}{2} - (c-1) \in \sigma(t+1,k-1).\]
  If $c=0$ and $b\geq 2$, then
  \[s+k = (t+1)(k-1) - \binom{a}{2} - \binom{b}{2} +1 = (t+1)(k-1)-\binom{a}{2}-\binom{b-1}{2}-(b-2)
    \in \sigma(t+1,k-1).\]
  If $c=0$ and $b\leq 1$, then
  \[s+k = t(k-1)-\binom{a}{2}+k =(t+1)k-\binom{a}{2}-t = (t+1)k-\binom{a+1}{2}-(t-a),\]
  and for
  \[\mathcal F=\left\{[k+1]\setminus\{i\}\,:\,i\in\{1,2,\dots,a+1\}\right\}\cup\left\{\{2,\dots,k\}\cup\{k+1+i\}\,:\,i\in[t-a]\right\}\]
  we obtain
  \[\abs{\Delta\mathcal F}=(a+1)k-\binom{a+1}{2}+(t-a)(k-1)=(t+1)k-\binom{a+1}{2}-(t-a)=s,\]
  and this concludes the proof of (i).

  For (ii), let $s_1 \in \sigma(t_1,k)$ and $s_2 \in
  \sigma(t_2,k-1)$. By~(\ref{eq:shadow_spectrum}), there are integers $a\geq b\geq c\geq 0$ and
  $a'\geq b'\geq c'\geq 0$ with $1\leq a+b\leq t_1$ , $1\leq a'+b'\leq t_2$,
  \[s_1 = t_1k - \binom{a}{2} - \binom{b}{2} - c \qquad \text{and} \qquad s_2 = t_2(k-1) -
    \binom{a'}{2} - \binom{b'}{2} - c'.\]
  First, we assume $a\geq 3$, and argue that without loss of generality $a+b-2c\geq 2$. If $a+b-2c=0$, then $a=b=c$, and from
  \[\binom{a}{2} + \binom{a}{2} + a = \binom{a+1}{2} + \binom{a-1}{2} + (a-1),\]
  it follows that we can use $(a+1,a-1,a-1)$ instead of $(a,a,a)$. Similarly, if $a+b-2c=1$, then
  $b=c=a-1$, and from
  \[\binom{a}{2} + \binom{a-1}{2} + a-1 = \binom{a+1}{2} + \binom{a-2}{2} + a-3,\]
  it follows that we can use $(a+1,a-2,a-3)$ instead of $(a,a-1,a-1)$.

  So we assume $a+b-2c \ge 2$ and show that there is a $(t_1+t_2)$-family $\mathcal{F}$ of $k$-sets
  such that $|\Delta \mathcal{F}| = s_1+s_2$. Without loss of generality, we assume that $b=0$ or
  $c\geq 1$. This is possible, because if $b\geq 1$ and $c=0$ then
  $\binom{b}{2}+c=\binom{b-1}{2}+(b-1)$, and we can use $(a,b-1,b-1)$ instead of $(a,b,0)$. If $b=0$
  then $a+a'+b'-c'\leq t_1+t_2\leq k+1$, and if $c\geq 1$ then
  $(a+b-c)+(a'+b'-c') \le (t_1-1) + t_2 \le k$. In both cases, by Remark~\ref{re:G(a,b,c,a',b',c')}
  there is a family $\mathcal F'=\mathcal{G}_k(a,b,c,a',b',c')\subseteq\binom{[k+2]}{k}$ with
  $\abs{\mathcal F'}=a+b+a'+b'$ and
    \[\abs{\Delta\mathcal F'}=(a+b+a'+b')k - \binom{a}{2} - \binom{b}{2} - c
      -\binom{a'}{2}-\binom{b'}{2}-c'-a'-b'.\]
    Adding to $\mathcal F'$ a $(t_1-a-b)$-family of
    $k$-sets with marginal shadow $k$, and a $(t_2-a'-b')$-family of $k$-sets with marginal shadow
    $k-1$ we obtain a $(t_1+t_2)$-family $\mathcal F$ of $k$-sets with
    \[\abs{\Delta\mathcal F}=\abs{\Delta\mathcal F'} + (t_1-a-b)k + (t_2-a'-b')(k-1)=s_1+s_2.\]

    It remains to consider the case $a\leq 2$. Then $t_1k-4 \leq s_1 \leq t_1k$. For the values
    $t_1k-l$, $l=4,3,2,1$, we use the same construction as above with
    $\mathcal F' = \mathcal{G}_k(3,1,1,a',b',c')$, $\mathcal{G}_k(3,0,0,a',b',c')$,
    $\mathcal{G}_k(2,1,1,a',b',c')$, and $\mathcal{G}_k(2,0,0,a',b',c')$, respectively. For
    $s_1=t_1k$ we use $\mathcal F'=\mathcal{G}_k(1,1,0,a',b',c')$ if $a'+b'\leq t_2-1$ or $c'\geq 1$
    (which ensures $2+a'+b'-c'\leq t_1+t_2-1\leq k$). If $a'+b'=t_2$ and $c'=0$, then
  \[s_1+s_2 = (t_1+t_2)k - \binom{a'}{2} - \binom{b'}{2} - (a'+b') = (t_1+t_2)k - \binom{a'+1}{2} -
    \binom{b'}{2} - b'\in\sigma(t_1+t_2,k),\]
  where we used $a'+1+b'=t_2+1\leq t_1+t_2 $.
\end{proof}
We use Lemma~\ref{lem:two_levels_suffice} to show that the antichains on three levels, which appear
in Lemma~\ref{lem:few_sets_out_of_big_level}, can be replaced by flat antichains.
\begin{lemma}\label{lem:two_levels}
  Let $n$ be a positive integer, $k=\lceil n/2\rceil$, let $\mathcal A\subseteq\binom{[n]}{k-1}\cup\binom{[n]}{k}\cup\binom{[n]}{k+1}$ be a
  maximal antichain, and set $t_1=\abs{\mathcal A\cap\binom{[n]}{k+1}}$,
  $t_2=\abs{\mathcal A\cap\binom{[n]}{k-1}}$. If $t_1+t_2\leq k+2$ then there exists
  $\mathcal F\subseteq\binom{[n]}{k+1}$ or $\mathcal F\subseteq\binom{[n]}{n-k+1}$ with
  \[\abs{\mathcal A}=\abs{\mathcal F}+\binom{n}{k}-\abs{\Delta\mathcal F}.\]
\end{lemma}
\begin{proof}
  Let $\mathcal A_1=\mathcal A\cap\binom{[n]}{k+1}$ and $\mathcal A_2=\mathcal
  A\cap\binom{[n]}{k-1}$. If $t_2=0$ then $\mathcal A=\mathcal
  F\cup\binom{[n]}{k}\setminus\Delta\mathcal F$ for $\mathcal F=\mathcal A_1$, and if $t_1=0$ then
  $\mathcal A=\mathcal F\cup\binom{[n]}{k}\setminus\Delta\mathcal F$ for $\mathcal F=\mathcal
  A\cap\binom{[n]}{k}$. So we assume $t_1\geq 1$ and $t_2\geq 1$, and then
  \[\abs{\mathcal A}=t_1+t_2+\abs{\binom{[n]}{k}\setminus\left(\Delta\mathcal A_1\cup\nabla\mathcal
        A_2\right)}=\binom{n}{k}+t_1+t_2-\abs{\Delta\mathcal A_1}-\abs{\nabla\mathcal A_2}.\]
  Since
  $\abs{\Delta\mathcal A_1}\in\sigma(t_1,k+1)$ and $\abs{\nabla\mathcal A_2}\in\sigma(t_2,n-k+1)$, it
  follows that $\abs{\Delta\mathcal A_1}+\abs{\nabla\mathcal A_2}\in\sigma(t_1+t_2,k+1)$ if $n=2k$, and with Lemma~\ref{lem:two_levels_suffice},
  \[\abs{\Delta\mathcal A_1}+\abs{\nabla\mathcal A_2}\in\sigma(t_1+t_2,k)\cup\sigma(t_1+t_2,k+1)\]
  if $n=2k-1$.
  As a consequence, there exists a
  $(t_1+t_2)$-family $\mathcal F$ with $\mathcal F\subseteq\binom{\nats}{k+1}$ or
  $\mathcal F\subseteq\binom{\nats}{k}$ and
  $\abs{\Delta\mathcal F}=\abs{\Delta\mathcal A_1}+\abs{\nabla\mathcal A_2}$. Using the second part
  of Theorem~\ref{thm:shadow_spectrum}, we can assume that the members of $\mathcal F$ are subsets of
  $[k+4]\subseteq[n]$, as required.
\end{proof}
Now we have all the ingredients to prove Lemma~\ref{lem:equivalent_to_large_sizes}.
\begin{proof}[Proof of Lemma~\ref{lem:equivalent_to_large_sizes}]
As $m\in S(n)$, there exists a maximal antichain $\mathcal A$ in $B_n$ with $|\mathcal A|=m$.
By Lemmas \ref{lem:A_is_in_the_middle_1} and \ref{lem:A_is_in_the_middle_2},
$\mathcal A\subseteq \binom{[n]}{k-1}\cup\binom{[n]}{k}\cup\binom{[n]}{k+1}$ or $n=2k-1$ and
$\mathcal A\subseteq \binom{[n]}{k-2}\cup\binom{[n]}{k-1}\cup\binom{[n]}{k}$.
In the latter case, the family consisting of the complements of the members of $\mathcal A$ is a maximal antichain
in $\binom{[n]}{k-1}\cup\binom{[n]}{k}\cup\binom{[n]}{k+1}$, i.e., without loss of generality, we
can assume that $\mathcal A\subseteq \binom{[n]}{k-1}\cup\binom{[n]}{k}\cup\binom{[n]}{k+1}$.
By Lemma \ref{lem:few_sets_out_of_big_level}, at most $k+1$ sets in $\mathcal A$ are not in $\binom{[n]}{k}$.
Now, by Lemma \ref{lem:two_levels}, there exists a family $\mathcal F\subseteq\binom{[n]}{l+1}$ with $l\in\{k,n-k\}$
such that $m=\binom nk+|\mathcal F|-|\Delta\mathcal F|$.
Finally, we have $|\mathcal F|\le k+1$ by Lemma \ref{lem:few_sets_out_of_big_level}.
\end{proof}

\section{The lower part of the spectrum}\label{sec:lower_part}

In this section we prove \Cref{thm:main_result}(ii). It will suffice by induction to
construct antichains of large sizes $m$. More specifically, we set
\[w(n)=\binom{n}{\lceil n/2\rceil}-\left\lceil\frac{n}{2}\right\rceil\left\lceil\frac{n+1}{2}\right\rceil\]
and establish the following claim.
\begin{claim}\label{lem:induction_step}
  If $n\geq 7$, then $[w(n-1)+2,w(n)]\subseteq S(n)$.
\end{claim}
Assuming this claim, the proof of Theorem~\ref{thm:main_result}(ii) is easy.
\begin{proof}[Proof of Theorem~\ref{thm:main_result}(ii) (assuming Claim~\ref{lem:induction_step})]
  We proceed by induction on $n$. As we want to use Claim~\ref{lem:induction_step} for the induction
  step, we have to establish the result for $n\leq 6$ as the base case.
  For $1\leq m\leq n$, a maximal antichain of size $m$ is given by
  $\{\{1\},\{2\},\dots,\{m-1\},\{m,\dots,n\}\}$. For $n\leq 5$, this is already sufficient, because
  then $w(n)\leq n$. As $w(6)=14$, we need maximal antichains of sizes
  $m=7,8,\dots,14$ for $n=6$. Such are given by:
  \begin{description}
  \item[$m=7$] $\mathcal A=\{\{1,2,3\},\{1,2,4\},\{1,2,5\},\{3,4\},\{3,5\},\{4,5\},\{6\}\}$,
  \item[$m=8$] $\mathcal A=\binom{[4]}{2}\cup\{\{5\},\{6\}\}$,
  \item[$m=9$]
    $\mathcal
    A=\{\{1,2,5\},\{1,2,6\},\{3,4,5\},\{3,4,6\},\{5,6\},\{1,3\},\{1,4\},\{2,3\},\{2,4\}\}$,
  \item[$m=10$]
    $\mathcal
    A=\{\{1,2\},\{1,3\},\{1,4\},\{2,5\},\{3,6\}\}\cup\binom{[6]}{3}\setminus\nabla\{\{1,2\},\{1,3\},\{1,4\},\{2,5\},\{3,6\}\}$,
  \item[$m=11$]
    $\mathcal
    A=\{\{1,2\},\{1,3\},\{1,4\},\{2,5\},\{3,5\}\}\cup\binom{[6]}{3}\setminus\nabla\{\{1,2\},\{1,3\},\{1,4\},\{2,5\},\{3,5\}\}$,
  \item[$m=12$]
    $\mathcal
    A=\{\{1,2\},\{2,3\},\{4,5\}\}\cup\binom{[6]}{3}\setminus\nabla\{\{1,2\},\{2,3\},\{4,5\}\}$,
  \item[$m=13$]
    $\mathcal
    A=\{\{1,2\},\{2,3\},\{3,4\}\}\cup\binom{[6]}{3}\setminus\nabla\{\{1,2\},\{2,3\},\{3,4\}\}$,
  \item[$m=14$]
    $\mathcal A=\{\{1,2\},\{3,4\}\}\cup\binom{[6]}{3}\setminus\nabla\{\{1,2\},\{3,4\}\}$.
  \end{description}
  For the induction step, we assume $n\geq 7$ and $[1,w(n-1)]\subseteq S(n-1)$. Noting that
  $1\in S(n)$ for every $n$, and adding the singleton $\{n\}$ to each of the maximal antichains with
  sizes $m\in[1,w(n-1)]$ in $B_{n-1}$, we obtain $[1,w(n-1)+1]\subseteq S(n)$, and we use
  Claim~\ref{lem:induction_step} to conclude $[1,w(n)]\subseteq S(n)$ as required.
\end{proof}

The rest of the section is devoted to the proof of Claim~\ref{lem:induction_step}. In
Subsection~\ref{sec:construction} we describe a construction for maximal antichains on three
consecutive levels $k-1,k,k+1$, we prove that this yields an interval $I(n,k)$ of maximal antichain
sizes (\Cref{lem:interval_n_k}), and we provide bounds for the endpoints of these intervals
(\Cref{lem:size_of_I(nk)}). In Section~\ref{sec:large_n} we show that for $n\geq 20$, the intervals
$I(n,k)$ for varying $k$ overlap and taking their union establishes
Claim~\ref{lem:induction_step}. Finally, in Section~\ref{sec:small_n} we use a separate construction
to fill the gaps between the intervals $I(n,k)$ for $7\leq n\leq 19$, thus completing the proof of
Claim~\ref{lem:induction_step}.

\subsection{The main construction}\label{sec:construction}
Before going into the details of the construction we outline the general idea. An antichain
$\mathcal A\subseteq\binom{[n]}{k}\cup\binom{[n]}{k-1}$ is a \emph{maximal squashed flat antichain}
if it has the form $\mathcal A=\mathcal F\cup\binom{{n}}{k-1}\setminus\Delta\mathcal F$ where
$\mathcal F\subseteq\binom{[n]}{k}$ is an initial segment in squashed order. From a maximal
squashed flat antichain $\mathcal A'\subseteq\binom{[n]}{k}\cup\binom{[n]}{k-1}$ with
$\binom{[k+3]}{k}\subseteq\mathcal A'$, we can obtain a new maximal antichain $\mathcal A$ in the
following way. We replace $\binom{[k+3]}{k}$ in $\mathcal{A}'$ by
$\mathcal{F}\cup\binom{[k+3]}{k}\setminus\Delta\mathcal{F}$ where
$\mathcal{F}\subseteq\binom{[k+3]}{k+1}$ contains only few sets (about $k/2$). Using
\Cref{thm:shadow_spectrum}, we can vary $\mathcal F$ without changing its cardinality to get $k$
consecutive values for the size of $\Delta \mathcal{F}$.  For a few small values of $n$ we only get
$k-1$ or $k-2$ consecutive shadow sizes, and we close the resulting gaps by extra constructions.
Hence, from every maximal squashed $\mathcal A'\subseteq\binom{[n]}{k}\cup\binom{[n]}{k-1}$, we
obtain an interval of $k$ consecutive sizes of maximal antichains in $B_n$. We then vary
$\mathcal A'$.  We say that two distinct maximal squashed antichains
$\mathcal{A}_1,\mathcal{A}_2\subseteq\binom{[n]}{k}\cup\binom{[n]}{k-1}$ are \emph{consecutive}\/ if
there is no maximal squashed antichain $\mathcal{A}_3\subseteq\binom{[n]}{k}\cup\binom{[n]}{k-1}$
with
$|\mathcal{A}_1\cap\binom{[n]}{k}|< |\mathcal{A}_3\cap\binom{[n]}{k}| <
|\mathcal{A}_2\cap\binom{[n]}{k}|$.  (Note that for consecutive $\mathcal{A}_1,\mathcal{A}_2$ it is
still possible that there is a maximal squashed antichain $\mathcal{A}_3$ with size $\abs{\mathcal
  A_3}$ between $\abs{\mathcal{A}_1}$ and $\abs{\mathcal{A}_2}$.) Observing that the sizes of consecutive maximal
squashed flat antichains differ by at most $k+1$, we deduce that the maximal antichains on levels
$k+1$, $k$ and $k-1$ obtained from all $\mathcal A'$ yield an interval $I(n,k)$ of sizes of
maximal antichains.

For an integer $k$ with $\lfloor n/2 \rfloor \le k \le n-3$ we denote by $\mathcal{M}(n,k)$ the set
of all maximal squashed flat antichains $\mathcal{A}$ with
$\binom{[k+3]}{k} \subseteq \mathcal{A} \subseteq \binom{[n]}{k}\cup \binom{[n]}{k-1}$. Moreover, we
let $Y(n,k)=\{\abs{\mathcal{A}}\,:\,\mathcal{A} \in \mathcal{M}(n,k)\}$ be the set of their
sizes. The following result is proved using the construction outlined above.
\begin{lemma}\label{lem:construction}
  Let $n$, $k$ and $t$ be integers with $n\geq 7$, $\frac{n-1}{2}\le k \le n-3$ and $1\leq
  t\leq k$. Then for every $m\in Y(n,k)$ and $s\in \sigma(t,k)$ there is a maximal antichain
  $\mathcal A \subseteq \binom{[k+3]}{k+1} \cup \binom{[n]}{k} \cup \binom{[n]}{k-1}$ with
  $\abs{\mathcal A}=m-s$.
\end{lemma}
\begin{proof}
  Fix $m\in Y(n,k)$, $s\in\sigma(t,k)$, and $\mathcal A'\in\mathcal{M}(n,k)$ with
  $\abs{\mathcal A'}=m$. By \Cref{thm:shadow_spectrum} and \Cref{rem:small_ground_set}, $s+t\in\sigma(t,k+1)$, and
  there is a family $\mathcal F\subseteq\binom{[k+3]}{k+1}$ with $\abs{\mathcal F}=t$ and
  $\abs{\Delta\mathcal F}=s+t$. By \Cref{lem:standard_ac_is_max}
  $\mathcal F'=\mathcal F\cup\binom{[k+3]}{k}\setminus\Delta\mathcal F$ is a maximal antichain in
  $B_{k+3}$. As a consequence,
  $\mathcal A=\left[\mathcal A'\setminus\binom{[k+3]}{k}\right]\cup\mathcal F'$ is a maximal
  antichain in $B_n$ with
  \[\abs{\mathcal A}=\abs{A'}-\binom{k+3}{k}+\abs{\mathcal F'}=m-\binom{k+3}{k}+\abs{\mathcal
      F}-\binom{k+3}{3}+\abs{\Delta\mathcal F}=m+t-(s+t)=m-s.\qedhere\]
\end{proof}
By \Cref{thm:shadow_spectrum} and \Cref{prop:interval_description}, for $t\leq k+1$,
\[\left[tk - \binom{t-j^*(t)}{2} - \binom{j^*(t)+1}{2},\,tk\right] \subseteq \sigma(t,k),\]
where $j^*(t)=\lceil \sqrt{2t}-5/2\rceil$ is the smallest non-negative integer $j$ with
$\binom{j+3}{2}\geq t$. Thus we have the following consequence of \Cref{lem:construction}.
\begin{corollary}\label{cor:mini_intervals}
  Let $n\geq 7$ and $k$ be integers with $\frac{n-1}{2}\le k \le n-3$. Let $m \in Y(n,k)$ and set
  $t=\lfloor \frac{k+3}{2} \rfloor$, $j=j^*(t)$. Then
  \[\left[m-tk,\,m-tk+\binom{t-j}{2}+\binom{j+1}{2}\right]\subseteq S(n).\]
\end{corollary}
In the next lemma we bound the lengths of the intervals in \Cref{cor:mini_intervals}.
\begin{lemma}\label{lem:last_shadow_interval}
  Let $k\geq 3$, $t=\left\lfloor\frac{k+3}{2}\right\rfloor$, $j=j^*(t)$, and set
  $L=\binom{t-j}{2}+\binom{j+1}{2}$.
  \begin{enumerate}[(i)]
  \item If $k\geq 9$ or $k\in\{3,7\}$, then $L\geq k$.
  \item If $k\in\{4,5,8\}$, then $L=k-1$.
  \item If $k=6$, then $L=4$.
  \end{enumerate}
\end{lemma}
\begin{proof}
   For $3\leq k\leq 32$, the statement can be verified by calculating $L$. We now assume $k\geq
   33$. Using $\frac{k+2}{2}\leq t\leq\frac{k+3}{2}$, we bound $j$ by
   \[j=\left\lceil\sqrt{2t}-\frac52\right\rceil\leq\sqrt{2t}\leq\sqrt{k+3}.\]
   Then
   \[
L\geq\binom{t-j}{2}\geq\frac12\left(\frac{k+2}{2}-\sqrt{k+3}\right)\left(\frac{k}{2}-\sqrt{k+3}\right)\geq\frac{1}{2}\left(\frac{k}{2}-\sqrt{k+3}\right)^2=\frac{k^2}{8}-\frac12k\sqrt{k+3}+\frac{k+3}{2},\]
   and
   \begin{multline*}
     L-k\geq\frac{k^2}{8}-\frac12k\sqrt{k+3}-\frac{k}{2}=\frac{k}{8}\left(k-4\sqrt{k+3}-4\right)\geq\frac{k}{8}\left(k-4\sqrt{k+3}-9\right)\\
    =\frac{k}{8}\left((k+3)-4\sqrt{k+3}-12\right) =\frac{k}{8}\left(\sqrt{k+3}+2\right)\left(\sqrt{k+3}-6\right)\geq 0,
   \end{multline*}
   where the last inequality follows from $k+3\geq 36$.
\end{proof}
To show that the union of the intervals from \Cref{cor:mini_intervals} is an interval, we need an
upper bound for the gaps between the elements of $Y(n,k)$. This is provided in the next lemma.
\begin{lemma}\label{lem:squashed_gaps}
  The sizes of any two consecutive maximal squashed flat antichains on levels $k$ and $k-1$ in $B_n$
  differ by at most $\max\{k-1,n-k\}$. In particular, for $k \ge \frac{n+1}{2}$ any two consecutive
  elements of $Y(n,k)$ differ by at most $k-1$, for $k=\frac{n}{2}$ they differ by at most $k$, and
  for $k=\frac{n-1}{2}$ they differ by at most $k+1$.
\end{lemma}
\begin{proof}
  Let $A_m$ be the $m$-th $k$-set in squashed order, and let $\mathcal F(k,m)=\{A_1,\dots,A_m\}$ be
  the collection of the first $m$ $k$-sets in squashed order. Let $\Delta' A_m$ denote the marginal
  shadow of $A_m$, that is, $\Delta' A_m =\Delta A_m\setminus\Delta\mathcal F(k,m-1)$. For
  $m<\binom{n}{k}$, the collection
  $\mathcal A(k,m,n)=\mathcal F(k,m)\cup\binom{[n]}{k-1}\setminus\Delta\mathcal F(k,m)$ is a maximal
  antichain in $B_n$ if and only if $\Delta' A_{m+1}\neq\emptyset$.
  Note that $\Delta' A_{m+1}\neq\emptyset$ if and only if $1\in A_{m+1}$ because if $1\in A_{m+1}$ then
  $A_{m+1}\setminus\{1\}\in\Delta'A_{m+1}$ and if $1\notin A_{m+1}$ and $x\in A_{m+1}$, then
  $A_{m+1}\setminus\{x\}\in \Delta(A_{m+1}\setminus\{x\}\cup\{1\})$.

  Suppose that $\mathcal A=\mathcal A(k,m,n)$ is a maximal squashed antichain and let
  $\mathcal A'=\mathcal A(k,m',n)$ be the next maximal squashed antichain. We will conclude the
  proof by verifying that $\abs{\mathcal A}-(k-1)\leq\abs{\mathcal A'}\leq\abs{\mathcal A}+(n-k)$.
  If $\Delta' A_{m+2}\neq\emptyset$ then $m'=m+1$, and the claim follows from $\abs{\mathcal A'}=\abs{\mathcal
  A}+1-\abs{\Delta' A_{m+1}}$.
  If $\Delta'A_{m+2}=\emptyset$, then $1\in A_{m+1}$, $2\notin A_{m+1}$ and $A_{m+2}=A_{m+1}\setminus\{1\}\cup\{2\}$.
  It follows that $\Delta'A_{m+1}=\{A_{m+1}\setminus\{1\}\}$. Moreover, with $y:=\min A_{m'}$ we have
  $A_{m+i}=A_{m'}\setminus\{y\}\cup\{i\}$ for $i=1,2,\dots,m'-m$. Hence, $2\le m-m'=y\le n-k-1$,
  $\abs{\Delta' A_{m+1}}=1$ and $\Delta'A_{m+2}=\dots=\Delta' A_{m'}=\emptyset$.
  The claim follows with $\abs{\mathcal A'}=\abs{\mathcal A}+(m'-m)-\abs{\Delta'A_{m+1}\cup\dots\cup\Delta'A_{m'}}$.
\end{proof}

\begin{lemma}\label{lem:interval_n_k}
  Let $n\geq 7$, $\lfloor n/2\rfloor\leq k\leq n-3$, $t=\lfloor\frac{k+3}{2}\rfloor$ and
  $j=j^*(t)$. Then $I(n,k)\subseteq S(n)$, where
  \[I(n,k)=\left[\min Y(n,k)-tk,\,\max Y(n,k)-tk+\binom{t-j}{2}+\binom{j+1}{2}\right].\]
\end{lemma}
\begin{proof}
  If $k\geq 9$ or $k\in\{3,7\}$, then by \Cref{lem:squashed_gaps,lem:last_shadow_interval}
  \[I(n,k) = \bigcup_{m \in Y(n,k)} \left[m-tk,\,m-tk+\binom{t-j}{2}+\binom{j+1}{2}\right],\]
   and the claim follows by \Cref{cor:mini_intervals}.
  For $k\in\{4,5,8\}$, the claim follows if
  $n\leq 2k$, and for $n=2k+1$ there is exactly one gap of length $k+1$ in $Y(n,k)$: the difference between
  the sizes of the last two maximal squashed flat antichains which have sizes $\binom{n}{k}-(k+1)$
  and $\binom{n}{k}$, respectively. As a consequence, the only element of $I(n,k)$ which is not
  covered by \Cref{cor:mini_intervals} is $\binom{n}{k}-tk-1$. For $k=6$, we still need to do
  something for $n=12$ and $n=13$. For $n=12$, we only need a maximal antichain of size
  $\binom{12}{6}-25$ by the same argument as above. For $n=13$, we need to look for the gaps in
  $Y(13,6)$ which have length $6$ or $7$. As before, there is only one gap of size 7. Writing
  down the elements of $Y(n,k)$ in the order in which they appear when we run through the antichains
  in $\mathcal M(n,k)$ ordered by the number of $k$-sets, then there are a few more differences of size
  6. Fortunately, in most cases these gaps are filled by maximal squashed antichains that occur
  earlier or later. Using a computer, we verify that when we write down the elements of $Y(n,k)$ in
  increasing order, the only gaps of size at least six appear between the three largest values, that is, between
  $\binom{13}{6}-13$, $\binom{13}{6}-7$ and $\binom{13}{6}$. As a consequence, we only need the sizes
  $\binom{13}{6}-25$, $\binom{13}{6}-26$ and $\binom{13}{6}-32$. Applying \Cref{lem:construction}
  with $\binom{n}{k}\in Y(n,k)$, it is sufficient to verify that the missing sizes can be written in
  the form $\binom{n}{k}-s$ for some $s\in\sigma(t',k)$ with $t'\leq k$. The proof is completed by observing that
  \begin{itemize}
   \item $\binom{9}{4}-13\in S(9)$ because $13\in\sigma(4,4)$,
   \item $\binom{11}{5}-21\in S(11)$ because $21\in\sigma(5,5)$,
   \item $\left\{\binom{13}{6}-25,\,\binom{13}{6}-26,\,\binom{13}{6}-32\right\}\subseteq
    S(13)$ because $\{25,26\}\subseteq\sigma(5,6)$ and $32\in\sigma(6,6)$,
   \item $\binom{12}{6}-25\in S(12)$ because $25\in\sigma(5,6)$,
   \item $\binom{17}{8}-41\in S(17)$ because $41\in\sigma(6,8)$.\qedhere
\end{itemize}
\end{proof}
We use results from~\cite{Griggs2021} to bound $\min Y(n,k)$. To state these results, let $C_l$ be
the sum of the first $l$ Catalan numbers: $C_l=\sum_{i=1}^l\frac{1}{i+1}\binom{2i}{i}$.
\begin{lemma}[Section 4.1 in \cite{Griggs2021}]\label{lem:results_order_paper} For $1<k<n$, the
  minimum size of a maximal squashed flat antichain is
  \[\min\left\{\binom{n}{k-1},\binom{n}{k}\right\}-C_{\min\{k-1,n-k\}}.\]
  The largest number of $k$-sets in a maximal squashed flat antichain of minimum size is
  $\binom{a_k}{k}+\dots+\binom{a_2}{2}$ where
  \begin{equation}\label{eq:last_min}
    a_i=
    \begin{cases}
      2i-2 & \text{for }i\leq n-k+1,\\
      n-k-1+i & \text{for }i\geq n-k+2.
    \end{cases}
  \end{equation}
\end{lemma}
From \Cref{lem:results_order_paper} we obtain the following values for $\min Y(n,k)$.
\begin{lemma}\label{lem:min_Y} If $5 \le k \le n-4$ then $\displaystyle\min Y(n,k)=\min\left\{\binom{n}{k-1},\binom{n}{k}\right\}-C_{\min\{k-1,n-k\}}$.
\end{lemma}
\begin{proof}
  To deduce this from \Cref{lem:results_order_paper}, we verify
  $\binom{a_k}{k}+\dots+\binom{a_2}{2}\geq\binom{k+3}{k}$, where the $a_i$ are given
  by~(\ref{eq:last_min}). If $k\leq n-k+1$ then
  $\binom{a_k}{k}+\dots+\binom{a_2}{2}\geq\binom{a_k}{k}=\binom{2k-2}{k}\geq\binom{k+3}{k}$, where
  the inequality follows from $k\geq 5$. If $k\geq n-k+2$, then
  $\binom{a_k}{k}+\dots+\binom{a_2}{2}\geq\binom{a_k}{k}=\binom{n-1}{k}\geq\binom{k+3}{k}$, where the
  inequality follows from $n-1\geq k+3$.
\end{proof}

We are now ready to express the intervals $I(n,k)$ in terms of $n$ and $k$.
\begin{lemma}\label{lem:size_of_I(nk)}
  Let $k \ge 5$, assume $\frac{n-1}{2} \le k \le n-4$, and set
  $t=\left\lfloor\frac{k+3}{2}\right\rfloor$, $j=j^*(t)$.
  \begin{enumerate}[(i)]
  \item If $k > \frac{n+1}{2}$ then
 \[I(n,k) \supseteq \left[ \binom{n}{k} - C_{n-k} -tk,\,\binom{n}{k-1} + \binom{k+3}{k} -
     \binom{k+3}{k-1} -tk + \binom{t-j}{2} + \binom{j+1}{2} \right].\]
\item If $k \le \frac{n+1}{2}$ then $\displaystyle I(n,k) = \left[ \binom{n}{k-1} - C_{k-1} -tk,\,\binom{n}{k} -tk + \binom{t-j}{2} + \binom{j+1}{2}\right]$.
\end{enumerate}
Moreover, $I(7,3)=[16,29]$, $I(8,4)=[41,61]$, and $I(9,4)=[69,117]$.
\end{lemma}
\begin{proof}
  For $k\geq 5$, by \Cref{lem:min_Y},
  \[\min I(n,k)=\min\left\{\binom{n}{k-1},\,\binom{n}{k}\right\}-C_{\min\{k-1,n-k\}}-tk,\]
  and this gives the left ends of the intervals in (i) and (ii). For $k>\frac{n+1}{2}$, the right
  end of the interval comes from
  $\binom{[k+3]}{k}\cup\binom{[n]}{k-1}\setminus\binom{[k+3]}{k-1}\in\mathcal M(n,k)$. For
  $k\leq\frac{n+1}{2}$, $\binom{[n]}{k}\in\mathcal M(n,k)$ implies
  \[\max I(n,k)= \binom{n}{k} -tk + \binom{t-j}{2} + \binom{j+1}{2}.\]
  This is also valid for $k\in\{3,4\}$, and gives the right ends of the intervals for
  $(n,k)\in\{(7,3),(8,4),(9,4)\}$. For the left ends of these intervals we do an exhaustive search
  over $\mathcal M(n,k)$ and find the following antichains. For $(n,k)=(7,3)$, the maximal squashed
  flat antichain with $21$ $3$-sets has size $25$. For $(n,k)=(8,4)$, the maximal squashed flat
  antichain with $37$ $4$-sets has size $53$. For $(n,k)=(9,4)$, the maximal squashed flat antichain
  with $37$ $4$-sets has size $81$.
\end{proof}

\subsection{Proof of Claim \ref{lem:induction_step} for large $n$}\label{sec:large_n}
In this section, we prove Claim \ref{lem:induction_step} for $n\geq 20$.
\begin{lemma}\label{lem:induction_large_n}
  If $n\geq 20$ then $[w(n-1)+2,w(n)]\subseteq S(n)$.
\end{lemma}
\begin{proof}
  For $20\leq n\leq 199$, we use \Cref{lem:size_of_I(nk),lem:interval_n_k} (and a computer) to
  verify the statement. For $n\geq 200$, we conclude using the following inequalities which are
  proved below:
  \begin{itemize}
  \item For $k=\left\lceil\frac{9n}{10}\right\rceil$, $\min I(n,k)<w(n-1)$ (\Cref{lem:lower_bound}).
  \item For $k=\lceil n/2\rceil$, $\max I(n,k)\geq w(n)$ (\Cref{lem:upper_bound}).
  \item For $\left\lceil\frac{n+2}{2}\right\rceil\leq
  k\leq\left\lceil\frac{9n}{10}\right\rceil$, $\max I(n,k)\geq\min I(n,k-1)$ (\Cref{lem:bridges}).\qedhere
  \end{itemize}
\end{proof}
\begin{lemma}\label{lem:lower_bound}
  For $n\geq 200$ and $k=\left\lceil\frac{9n}{10}\right\rceil$, $\min I(n,k)<w(n-1)$.
\end{lemma}
\begin{proof}
  By \Cref{lem:size_of_I(nk)}, we have to show
  \[\binom{n}{k}-C_{n-k}-\left\lfloor\frac{k+3}{2}\right\rfloor k\leq w(n-1).\]
  We prove the stronger inequality
  $\binom{n}{k}\leq\binom{n-1}{\left\lfloor\frac{n-1}{2}\right\rfloor}-n^2$. We bound the left-hand side by
  \[\binom{n}{k}=\binom{n}{n-k}\leq\left(\frac{en}{n-k}\right)^{n-k}\leq\left(\frac{200e}{19}\right)^{n/10}<1.4^n.\]
  For the right-hand side, we start with
  \[\binom{n-1}{\left\lfloor\frac{n-1}{2}\right\rfloor}\geq 2^{\left\lfloor(n-1)/2\right\rfloor}\geq
    2^{\frac{99n}{200}}.\]
  Combining this with $\displaystyle n^2\leq 2^{\frac{98n}{200}}\leq 2^{\frac{98n}{200}}\left(2^{\frac{n}{200}}-1\right)$,
  we obtain $\displaystyle\binom{n-1}{\left\lfloor\frac{n-1}{2}\right\rfloor}-n^2\geq 2^{\frac{49n}{100}}>1.4^n$.
\end{proof}
\begin{lemma}\label{lem:upper_bound}
  For $n\geq 200$ and $k=\lceil n/2\rceil$, $\max I(n,k)\geq w(n)$.
\end{lemma}
\begin{proof}
  Set $t=\left\lfloor\frac{k+3}{2}\right\rfloor$ and $j=j^*(t)$. By \Cref{lem:size_of_I(nk)}, we
  have to show
  \[\binom{n}{k}-tk+\binom{t-j}{2}+\binom{j+1}{2}\geq \binom{n}{k}-k\left\lceil\frac{k+1}{2}\right\rceil.\]
  We prove the stronger inequality $\displaystyle\binom{n}{k}-tk+\binom{t-j}{2}\geq\binom{n}{k}-k\left\lceil\frac{k+1}{2}\right\rceil$,
  or, equivalently,
  \[\binom{t-j}{2}\geq
    k\left(\left\lfloor\frac{k+3}{2}\right\rfloor-\left\lceil\frac{k+1}{2}\right\rceil\right).\]
  If $k$ is even then the right-hand side is $0$, and there is nothing to do. For odd $k$, the
  right-hand side is equal to $k$, and we can show $\binom{t-j}{2}\geq k$ as in the proof of
  \Cref{lem:last_shadow_interval}. The assumption $k\geq 33$ made in this proof follows from $n\geq 200$.
\end{proof}
\begin{lemma}\label{lem:bridges}
  For $n\geq 200$ and $\left\lceil\frac{n+2}{2}\right\rceil\leq k\leq\left\lceil\frac{9n}{10}\right\rceil$, $\max I(n,k)\geq\min I(n,k-1)$.
\end{lemma}
\begin{proof}
 Set $t=\left\lfloor\frac{k+3}{2}\right\rfloor$, $t'=\left\lfloor\frac{k+2}{2}\right\rfloor$ and
 $j=j^*(t)$. By \Cref{lem:size_of_I(nk)}, we have to show that
 \[\binom{n}{k-1} + \binom{k+3}{k} -
     \binom{k+3}{k} -tk + \binom{t-j}{2} + \binom{j+1}{2}\geq\min I(n,k-1).\]
  For $k=\frac{n+2}{2}$, this is
  \[\binom{n}{k-1}+\binom{k+3}{k}-\binom{k+3}{k-1}-tk+\binom{t-j}{2}+\binom{j+1}{2}\geq\binom{n}{k-2}-C_{k-2}-t'(k-1).\]
  From $n^5\leq\binom{n}{k-1}$, we obtain
  \[n^4\leq\frac1k\binom{n}{k-1}=\binom{n}{k-1}-\binom{n}{k-2},\]
  and this can be used to bound the left-hand side:
  \[\binom{n}{k-1}+\binom{k+3}{k}-\binom{k+3}{k-1}-tk+\binom{t-j}{2}+\binom{j+1}{2}\geq\binom{n}{k-1}-n^4\geq\binom{n}{k-2},\]
  which is obviously larger than the right-hand side. Now we assume $k\geq\left\lfloor\frac{n+4}{2}\right\rfloor$, and the claim will follow from
  \[\binom{n}{k-1}+\binom{k+3}{k}-\binom{k+3}{k-1}-tk+\binom{t-j}{2}+\binom{j+1}{2}
    \geq\binom{n}{k-1}-C(n-k+1)-t'(k-1).\]
  The left-hand side is larger than $\binom{n}{k-1}-n^4$ and the right-hand side is smaller than
  \[\binom{n}{k-1}-C_{n-k+1}\leq\binom{n}{k-1}-C_{\lfloor n/10\rfloor}\leq\binom{n}{k-1}-\frac{1}{\lfloor n/10\rfloor+1}\binom{2\lfloor n/10\rfloor}{\lfloor n/10\rfloor}.\]
    Therefore, it is sufficient to
  verify $\frac{1}{l+1}\binom{2l}{l}\geq n^4$ for $l=\lfloor n/10\rfloor$. We use the bound
  \[\frac{1}{l+1}\binom{2l}{l}\geq \frac{4^l}{(l+1)\left(\pi l\frac{4l}{4l-1}\right)^{1/2}}\]
  from~\cite{Dutton1986}. From $n\geq 200$, it follows that $l\geq 20$, and this implies $4^l\geq
  25000l^{11/2}$. Bounding the denominator by
  \[(l+1)\left(\pi l\frac{4l}{4l-1}\right)^{1/2}\leq\frac{23l}{22}\left(\pi
      l\frac{88}{87}\right)^{1/2}\leq 2l^{3/2},\]
  we obtain
  \[\frac{4^l}{(l+1)\left(\pi l\frac{4l}{4l-1}\right)^{1/2}}\geq 12500 l^4\geq (11l)^4\geq n^4,\]
  which concludes the proof.
\end{proof}

\subsection{Proof of Claim \ref{lem:induction_step} for small $n$}\label{sec:small_n}
In this Section, we prove Claim \ref{lem:induction_step} for $7\leq n\leq 19$. We need a few auxiliary results. Let
\[S(n,k)=\left\{\abs{\mathcal A}\,:\,\mathcal A\subseteq\binom{[n]}{k-1}\cup\binom{[n]}{k}\text{ is
      a maximal antichain in }B_n\right\}.\]

\begin{lemma}\label{lem:recursion}
  For $k\geq 2$ and $n\geq k+1$, $\displaystyle S(n,k)\supseteq S(n-1,k)+\binom{n-1}{k-2}$ and
  $\displaystyle S(n,k)\supseteq S(n-1,k-1)+\binom{n-1}{k}$.
\end{lemma}
\begin{proof}
  Let $\mathcal A'\subseteq\binom{[n-1]}{[k-1]}\cup\binom{[n-1]}{[k]}$ be a maximal antichain in
  $B_{n-1}$. Then
  \[\mathcal A=\mathcal
    A'\cup\left\{A\cup\{n\}\,:\,A\in\binom{[n-1]}{k-2}\right\}\subseteq\binom{[n]}{k-1}\cup\binom{[n]}{k}\]
  is a maximal antichain in $B_n$, and this implies the first inclusion. For the second one, let
  $\mathcal A'\subseteq\binom{[n-1]}{[k-2]}\cup\binom{[n-1]}{[k-1]}$ be a maximal antichain in
  $B_{n-1}$. Then
  \[\mathcal A=\left\{A\cup\{n\}\,:\,A\in\mathcal
      A'\right\}\cup\binom{[n-1]}{k}\subseteq\binom{[n]}{k-1}\cup\binom{[n]}{k}\]
  is a maximal antichain in $B_n$.
\end{proof}

\begin{definition}
  A flat antichain $\mathcal A\subseteq\binom{[n]}{k-1}\cup\binom{[n]}{k}$ is called
  \emph{$\{1,2\}$-separated} if $\{1,2\}\subseteq A$ for every $k$-set $A\in\mathcal A$
  $\lvert A\cap\{1,2\}\rvert\leq 1$ for every $(k-1)$-set $A\in\mathcal A$.
\end{definition}

\begin{observation}\label{obs:12_separated_recursion}
  Let $\mathcal A'\subseteq\binom{[n-1]}{k-1}\cup\binom{[n-1]}{k}$ and
  $\mathcal A''\subseteq\binom{[n-1]}{k-2}\cup\binom{[n-1]}{k-1}$ be $\{1,2\}$-separated
  maximal antichains in $B_{n-1}$.  Then
  $\mathcal A=\mathcal A'\cup\{A\cup\{n\}\,:\,A\in\mathcal
  A''\}\subseteq\binom{[n-1]}{k-1}\cup\binom{[n-1]}{k}$ is a $\{1,2\}$-separated maximal antichain
  in~$B_n$.
\end{observation}
This can be used in an induction to establish the following result.
\begin{lemma}\label{lem:separated_intervals}
  Let $n$, $k$ and $m$ be integers with $k\geq 2$, $n\geq k+1$, and
  \[\binom{n-1}{k-1}\leq m\leq \binom{n}{k-1}-2\binom{n-3}{k-3}-\binom{n-4}{k-5}.\]
  Then there exists a $\{1,2\}$-separated antichain
  $\mathcal A\subseteq\binom{[n]}{k-1}\cup\binom{[n]}{k}$ with $\abs{\mathcal A}=m$.
\end{lemma}
\begin{proof}
  For $k=2$, we have $m\in\{n-1,n\}$. Then $\binom{[n]}{1}$ is a $\{1,2\}$-separated maximal
  antichain of size $n$, and $\{\{1,2\}\}\cup\left(\binom{[n]}{1}\setminus\{\{1\},\,\{2\}\}\right)$
  is a $\{1,2\}$-separated maximal antichain of size $n-1$. For $k\geq 3$, $n=k+1$, we have
  $k\leq m\leq 2k-2$. For $l\in\{3,4,\dots,n\}$,
  \[\mathcal A= \left\{[n]\setminus\{j\}\,:\,3\leq j\leq l\right\}\ \cup\ \{\{3,4,\dots,n\}\}\ \cup\
    \{\{i\}\cup([3,n]\setminus\{j\})\,:\,i\in\{1,2\},\,j\in\{l+1,\dots,n\}\}\] is a maximal
  $\{1,2\}$-separated antichain with $\abs{\mathcal A}=l-2+2(n-l)+1=2k-l+1$. For $n\geq k+2$, we
  proceed by induction, and note that \Cref{obs:12_separated_recursion} implies that we
  get all sizes $m$ with
  \[\binom{n-2}{k-1}+\binom{n-2}{k-2}\leq
    m\leq\binom{n-1}{k-1}-2\binom{n-4}{k-3}-\binom{n-5}{k-5}+\binom{n-1}{k-2}-2\binom{n-4}{k-4}-\binom{n-5}{k-6},\]
  and this simplifies to the claimed range.
\end{proof}
\begin{lemma}\label{lem:bridges_2}
  For $\displaystyle 3\leq k\le 10$ and $2k\leq n\leq 20$, $\displaystyle\left[\binom{n-1}{k-1},\binom{n}{k-1}\right]\subseteq S(n,k)$.
\end{lemma}
\begin{proof}
  We proceed by induction on $k$, and for fixed $k$ by induction on $n$. For the base case
  $(k,n)=(3,6)$, we have to verify $[10,15]\subseteq S(6,3)$. For $m\in[10,14]$ we can use the
  maximal antichains on levels $2$ and $3$ in $B_6$ that are given in the proof of
  \Cref{thm:main_result}(ii) just before \Cref{sec:construction}, and for $m=15$ we use the maximal
  antichain $\binom{[6]}{2}$. Next, we look at the induction step from $(k,n-1)$ to $(k,n)$ for
  $n\geq 2k+1$. From the induction hypothesis and \Cref{lem:recursion},
  \[S(n,k)\supseteq\left[\binom{n-2}{k-1},\,\binom{n-1}{k-1}\right]+\binom{n-1}{k-2}=\left[\binom{n-2}{k-1}+\binom{n-1}{k-2},\,\binom{n}{k-1}\right],\]
  and the claim follows with \Cref{lem:separated_intervals} and
  \[\binom{n}{k-1}-2\binom{n-3}{k-3}-\binom{n-4}{k-5}\geq\binom{n-2}{k-1}+\binom{n-1}{k-2}\]
  for all $(k,n)$ with $4\leq k\leq 10$ and $2k+1\leq n\leq 20$. To complete the induction argument, we check that the claim for $k\geq 4$ and $n=2k$ is implied by
  the statement for $k-1$. We need four ingredients:
  \begin{itemize}
  \item By \Cref{lem:separated_intervals},
    \[\left[\binom{2k-1}{k-1},\,
        \binom{2k}{k-1}-2\binom{2k-3}{k-3}-\binom{2k-4}{k-5}\right]\subseteq S(2k,k).\]
  \item By \Cref{lem:recursion,lem:separated_intervals},
    \[S(2k,k)\supseteq\left[\binom{2k-2}{k-1},\,\binom{2k-1}{k-1}-2\binom{2k-4}{k-3}-\binom{2k-5}{k-5}\right]+\binom{2k-1}{k-2}.\]
  \item By \Cref{lem:recursion}, $S(2k,k) \supseteq S(2k-1,k) + \binom{2k-1}{k-2} \supseteq S(2k-2,k-1)+\binom{2k-2}{k}+\binom{2k-1}{k-2}$,
    and with induction,
    \[S(2k,k)\supseteq\left[\binom{2k-3}{k-2},\,\binom{2k-2}{k-2}\right]+\binom{2k-2}{k}+\binom{2k-1}{k-2}.\]
  \item By \Cref{lem:recursion} and induction,
    \[S(2k,k)\supseteq\left[\binom{2k-2}{k-2},\,\binom{2k-1}{k-2}\right]+\binom{2k-1}{k}=\left[\binom{2k-2}{k-2}+\binom{2k-1}{k},\,\binom{2k}{k-1}\right].\]
  \end{itemize}
  We check that for $4\leq k\leq 10$, the union of these four intervals is
  $[\binom{2k-1}{k-1},\,\binom{2k}{k-1}]$, as required.
\end{proof}
We need one more simple construction before we can complete the argument for $n\leq 19$.
\begin{lemma}\label{lem:final_case}
  Let $n \ge 2$ be an integer. If $m \in S(n-1)$ and $m > \binom{n-2}{\lfloor \frac{n-2}{2} \rfloor} +1$, then $m+n-1 \in S(n)$.
\end{lemma}
\begin{proof}
  Let $m \in S(n-1)$ with $m > \binom{n-2}{\lfloor \frac{n-2}{2} \rfloor} +1$, and $\mathcal{A}$ be
  a maximal antichain in $B_{n-1}$ with size $m$. Then, by Sperner's Theorem, $\mathcal{A}$ cannot
  contain a singleton. Thus, $\mathcal{A} \cup \{\{i,n\}\,:\, i \in [n-1]\}$ is a maximal antichain in
  $B_n$ of size $m+n-1$.
\end{proof}

Finally, we combine \Cref{lem:interval_n_k,lem:size_of_I(nk),lem:bridges_2,lem:final_case} to prove
Claim \ref{lem:induction_step} for $7\leq n\leq 19$.
\begin{lemma}\label{lem:induction_step_small_n}
  For all $n\in\{7,\dots,19\}$, $\left[w(n-1)+2,w(n)\right] \subseteq S(n)$.
\end{lemma}
\begin{proof} For each value of $n$ we list the sizes coming from the various lemmas.
  \begin{description}
  \item[$n=7$] We need the interval $[w(6)+2,w(7)]=[16,23]$, and \Cref{lem:size_of_I(nk)} yields
    $I(7,3)=[16,29]$.
  \item[$n=8$] We need $[w(7)+2,w(8)]=[25,58]$. From \Cref{lem:size_of_I(nk)} we have $I(8,4)=[41,61]$, and from \Cref{lem:bridges_2},
    $\left[\binom{7}{2},\binom{8}{2}\right]\cup\left[\binom{7}{3},\binom{8}{3}\right]=[21,28]\cup[35,56]$. Finally,
    Lemma~\ref{lem:final_case} with $[22,27]\subseteq S(7)$ yields $[29,34]\subseteq S(8)$.
  \item[$n=9$] We need $[w(8)+2,w(9)]=[60,111]$. From \Cref{lem:size_of_I(nk)} we have $I(9,4)=[69,117]$, and from \Cref{lem:bridges_2},
    $\left[\binom{8}{3},\binom{9}{3}\right]=[56,84]$.
  \item[$n=10$] We need $[w(9)+2,w(10)]=[113,237]$. From \Cref{lem:size_of_I(nk)}
    we have
    \[I(10,5)\cup I(10,6)=[164,236],\]
    and from \Cref{lem:bridges_2},
    $\left[\binom{9}{3},\binom{10}{3}\right]\cup\left[\binom{9}{4},\binom{10}{4}\right]=[84,120]\cup[126,210]$. From
    \Cref{lem:final_case} and $[112,116]\subseteq S(9)$, we deduce $[121,125]\subseteq S(10)$, and
    finally, a maximal antichain of size $237$ is given by $\mathcal A=\mathcal
    F\cup\binom{[10]}{5}\setminus\Delta\mathcal F$ where $\mathcal F=\{\{1, 2, 3, 4, 5, 6\},\,\{1, 2, 3, 4, 7, 8\},\,\{1, 2, 3, 4, 9, 10\}\}$.
  \item[$n=11$] We need $[w(10)+2,w(11)]=[239,438]$. From \Cref{lem:size_of_I(nk)}
    we have
    \[I(11,5)\cup I(11,7)=[273,446],\]
    and from \Cref{lem:bridges_2}, $\left[\binom{10}{4},\binom{11}{4}\right]=[210,330]$.
  \item[$n=12$] We need $[w(11)+2,w(12)]=[440,900]$. From \Cref{lem:size_of_I(nk)}
    we have
    \[I(12,6)\cup I(12,7)=[693,904],\]
    and from \Cref{lem:bridges_2}, $\left[\binom{11}{4},\binom{12}{4}\right]\cup\left[\binom{11}{5},\binom{12}{5}\right]=[330,792]$.
  \item[$n=13$] We need $[w(12)+2,w(13)]=[902,1688]$. From \Cref{lem:size_of_I(nk)}
    we have
    \[I(13,7)\cup I(13,8)=[1138,1688],\]
    and from \Cref{lem:bridges_2}, $\left[\binom{12}{5},\binom{13}{5}\right]=[792,1287]$.
  \item[$n=14$] We need $[w(13)+2,w(14)]=[1690,3404]$. From \Cref{lem:size_of_I(nk)}
    we have
    \[I(14,7)\cup I(14,8)=[2767,3404],\]
    and from \Cref{lem:bridges_2}, $\left[\binom{13}{5},\binom{14}{5}\right]\cup\left[\binom{13}{6},\binom{14}{6}\right]=[1287,3003]$.
  \item[$n=15$] We need $[w(14)+2,w(15)]=[3406,6395]$. From \Cref{lem:size_of_I(nk)}
    we have
    \[I(15,8)\cup I(15,9)=[4755,6402],\]
    and from \Cref{lem:bridges_2}, $\left[\binom{14}{6},\binom{15}{6}\right]=[3003,5005]$.
  \item[$n=16$] We need $[w(15)+2,w(16)]=[\num{6397},\num{12830}]$. From \Cref{lem:size_of_I(nk)}
    we have
    \[I(16,8)\cup I(16,10)=[\num{7752},\num{12837}],\]
    and from \Cref{lem:bridges_2}, $\left[\binom{15}{6},\binom{16}{6}\right]=[\num{5005}, \num{8008}]$.
  \item[$n=17$] We need $[w(16)+2,w(17)]=[\num{12832},\num{24265}]$. From \Cref{lem:size_of_I(nk)}
    we have
    \[ I(17,9)\cup I(17,10)=[\num{18763},\num{24267}],\]
    and from \Cref{lem:bridges_2}, $\left[\binom{16}{7},\binom{17}{7}\right]=[\num{11440}, \num{19448}]$.
  \item[$n=18$] We need $[w(17)+2,w(18)]=[\num{24267},\num{48575}]$. From \Cref{lem:size_of_I(nk)}
    we have
    \[ I(18,9)\cup I(18,10)\cup I(18,11)=[\num{31122},\num{48577}],\]
    and from \Cref{lem:bridges_2}, $\left[\binom{17}{7},\binom{18}{7}\right]=[\num{19448}, \num{31824}]$.
  \item[$n=19$] We need $[w(18)+2,w(19)]=[\num{48577},\num{92318}]$. From \Cref{lem:size_of_I(nk)}
    we have
    \[ I(19,10)\cup I(19,11)\cup I(19,12)=[\num{49679},\num{92329}],\]
    and from \Cref{lem:bridges_2}, $\left[\binom{18}{7},\binom{19}{7}\right]=[\num{31824}, \num{50388}]$.\qedhere
  \end{description}
\end{proof}

\section{The largest missing value in the shadow spectrum}\label{sec:asymptotics}
In this section, we prove \Cref{thm:asymptotics}. Recall that $j^*(t)$ is defined as the
smallest non-negative integer $j$ with $\binom{j+3}{2}\geq t$ and that $j^*(t)=\left\lceil\sqrt{2t}-5/2\right\rceil$ by \Cref{lem:jstar}.

Next, we make the observation that every integer $s\geq k^2$ is the shadow size for some
family of $k$-sets. Recall that $\Sigma(k)=\bigcup_{t=1}^\infty\sigma(t,k)$ is the set of all these
shadow sizes.
\begin{lemma}\label{lem:all_large_s}
  If $k\geq 3$ and $s$ are integers with $s\geq k^2$ then $s\in\Sigma(k)$.
\end{lemma}
\begin{proof}
  Let $\mathcal F$ be a family of $k$-sets with $\abs{\mathcal F}=t$, and let $A$ be a $k$-set that
  contains at least two elements not in any of the members of $\mathcal F$. Then $\mathcal
  F'=\mathcal F\cup\{A\}$ is a family of $k$-sets with $\abs{\mathcal F'}=t+1$ and
  $\abs{\Delta\mathcal F'}=\abs{\Delta\mathcal F}+k$. As a consequence,
  $\sigma(t+1,k)\supseteq\sigma(t,k)+k$. If we can show that $[k^2-k,k^2]\subseteq\sigma(k,k)$, then
  it follows by induction on $t$ that $[(t-1)k,tk]\subseteq\sigma(t,k)$ for all $t\geq k$, and this
  is clearly sufficient for our claim. By \Cref{thm:shadow_spectrum} and \Cref{prop:interval_description},
  \[\left[k^2-\binom{k-j^*(k)}{2}-\binom{j^*(k)+1}{2},k^2\right]\subseteq\sigma(k,k),\]
  and it suffices to verify $\binom{k-j^*(k)}{2}+\binom{j^*(k)+1}{2}\geq k$. For $k\leq 7$, this
  can be checked by hand, and for $k\geq 8$ we have $\sqrt{k}\geq2\sqrt2$, and with \Cref{lem:jstar}, 
  \[\binom{k-j^*(k)}{2}-k\geq\frac12\left(k-\sqrt{2k}\right)\left(k-\sqrt{2k}-1\right)-k\geq\frac12k^2-\sqrt2k^{3/2}\geq\frac12k^{3/2}(k-2\sqrt2)\geq 0.\qedhere\]
\end{proof}

Combining \Cref{thm:shadow_spectrum}, \Cref{prop:interval_description} and \Cref{lem:all_large_s},
the value $\psi(k)=\max\nats\setminus\Sigma(k)$ is the integer immediately before some set
$\{tk-x\,:\,x\in I_{j^*(t)}(t)\}$, and the $t$ must be chosen so that this number is too big to be
the shadow size of a $(t-1)$-family of $k$-sets. In other words, taking $t$ to be maximal with the
property that $tk-\binom{t-j^*(t)}{2}-\binom{j^*(t)+1}{2}>(t-1)k+1$, we will argue that (for $k\geq 21$)
$\psi(k)=tk-\binom{t-j^*(t)}{2}-\binom{j^*(t)+1}{2}-1$. Equivalently, we are looking for the
smallest $t^*=t^*(k)$ such that there is no gap between $\{tk-x\,:\,x\in I_{j^*(t^*)}(t^*)\}$ and
$\{(t+1)k-x\,:\,x\in I_{j^*(t^*+1)}(t^*+1)\}$. To be more precise, we define
\[f(t) = \binom{t-j^*(t)}{2}+\binom{j^*(t)+1}{2},\]
so that $I_{j^*(t)}(t)=\left[0,f(t)\right]$, and then
\[t^*=t^*(k) = \min\left\{t\,:\,(t+1)k-f(t+1)\leq tk+1\right\}= \max\left\{t\,:\,f(t)\leq
    k-2\right\}.\]
\begin{example}
  For $k=50$, we have
  \begin{align*}
    \sigma(10,50) &= [455, 455]\cup [463, 464]\cup [469, 500],\\
    \sigma(11,50) &= [495, 495]\cup [504, 505]\cup [511, 514]\cup [516, 550],\\
    \sigma(12,50) &= [534, 534]\cup [544, 545]\cup [552, 555]\cup [558, 600],\\
    \sigma(13,50) &= [572, 572]\cup [583, 584]\cup [592, 595]\cup [599, 650],\\
    \sigma(14,50) &= [609, 609]\cup [621, 622]\cup [631, 634]\cup [639, 700].
  \end{align*}
  We see that $t^*(50)=12$ and $\psi(50)=557$.
\end{example}
For $t=t^*(k)$ and $j=j^*(t)$, we want to show that $\psi(k)=tk-f(t)-1$. One part of the argument is
the verification that $tk-f(t)-1\not\in\sigma(t+1,k)$. For this step we will need the following
bound.
\begin{lemma}\label{lem:nonoverlap}
  For $k\geq 374$ and $t=t^*(k)$, $tk-f(t)\leq(t+1)k-\binom{t+1}{k}$.
\end{lemma}
\begin{proof}
  Rearranging terms, we need to show $k+f(t)\geq\binom{t+1}{2}$. Lemma~\ref{lem:jstar} implies
  the bound
  \begin{equation}\label{eq:fbound}
    f(t)\geq\binom{t-j^*(t)}{2}\geq\frac12\left(t-j^*(t)\right)\left(t-j^*(t)-1\right)\geq\frac12t^2-\sqrt2t^{3/2}.
  \end{equation}
  From $k\geq 374$ we obtain $t\geq 34$ and $\sqrt{t}\geq2\sqrt2+3$, which implies
  \[\frac12t^2-2\sqrt2t^{3/2}-\frac12t=\frac12t\left(t^{1/2}-2\sqrt2-3\right)\left(t^{1/2}-2\sqrt2+3\right)\geq
    0.\]
  We rearrange terms, use $f(t)\leq k$ and (\ref{eq:fbound}):
  \[\sqrt2t^{3/2}+\frac12t\leq\frac12t^2-\sqrt2t^{3/2}\leq f(t)\leq k,\]
  hence $k-\sqrt2t^{3/2}\geq t/2$. Finally, using~(\ref{eq:fbound}) again,
  \[k+f(t)\geq k+\frac12t^2-\sqrt2t^{3/2}\geq \frac12t^2+\frac12t=\binom{t+1}{2},\]
  as required.
\end{proof}

The next lemma is the precise version of the above description of $\psi(k)$ in terms of $t^*=t^*(k)$.
\begin{lemma}\label{lem:psi}
  For $k\geq 374$ and $t^*=t^*(k)$, $\psi(k)=t^*k-f(t^*)-1$.
\end{lemma}
\begin{proof}
  Let $s=t^*k-f(t^*)-1$. From $f(t^*)\leq k-2$ it follows that $s>(t^*-1)k$, hence
  $s\not\in\bigcup_{t=1}^{t^*-1}\sigma(t,k)$. Clearly, $s\not\in\sigma(t^*,k)$, and moreover
  Lemma~\ref{lem:nonoverlap} implies $s<\min\sigma(t,k)$ for all $t\geq t^*+1$. We
  conclude $s\not\in\Sigma(k)$. From $(t+1)k-f(t+1)\leq tk+1$ for all $t\geq t^*$ it follows that
  \[\left[t^*k-\binom{t^*-j^*\left(t^*\right)}{2}-\binom{j^*\left(t^*\right)+1}{2},k^2\right]\subseteq\Sigma(k).\qedhere\]
\end{proof}
To complete the proof of Theorem~\ref{thm:asymptotics}(\ref{asymptotics:item_i}), we need to bound
the expression for $\psi(k)$ provided in Lemma~\ref{lem:psi}. We start with the function $f$.
\begin{lemma}\label{lem:f_asym}
  $\displaystyle f(t) =\frac12t^2-\sqrt2t^{3/2}+O(t)$.
\end{lemma}
\begin{proof}
  By Lemma~\ref{lem:jstar}, $\sqrt{2t}-5/2<j^*(t)<\sqrt{2t}-3/2$, hence
  \[\binom{t-\sqrt{2t}+3/2}{2}+\binom{\sqrt{2t}-1/2}{2}<\binom{t-j^*(t)}{2}+\binom{j^*(t)+1}{2}<\binom{t-\sqrt{2t}+5/2}{2}+\binom{\sqrt{2t}-3/2}{2},\]
  and this implies
  \[\binom{t-j^*(t)}{2}+\binom{j^*(t)+1}{2}=\frac12\left(t-\sqrt{2t}\right)^2+O(t)=\frac12t^2-\sqrt2 t^{3/2}+O(t).\qedhere\]
\end{proof}
We set $t^*(k)=\max\left\{t\,:\,f(t)\leq k-2\right\}$, and look at the asymptotics of this function.
\begin{lemma}\label{lem:tstar_asymp}
  $\displaystyle t^*(k)=\sqrt{2k}+\sqrt[4]{8k}+O(1)$.
\end{lemma}
\begin{proof}
  From Lemma~\ref{lem:f_asym} it follows that $f(t+1)=\frac12t^2-\sqrt2t^{3/2}+O(t)$, and since by
  definition
  \[f\left(t^*(k)\right)\leq k-2<f\left(t^*(k)+1\right),\]
  we deduce
  \[\frac12t^*(k)^2-\sqrt2\,t^*(k)^{3/2}+O\left(t^*(k)\right)=k.\]
  This implies $t^*(k)=\sqrt{2k}+g(k)$ with $g(k)=o\left(\sqrt k\right)$, and
  \begin{equation}\label{eq:target}
    \frac12t^*(k)^2-\sqrt2\,t^*(k)^{3/2}=k+O\left(\sqrt k\right).
  \end{equation}
  The left-hand side is
  \begin{multline*}
    \frac12t^*(k)^2-\sqrt2\,t^*(k)^{3/2}=\frac12\left(\sqrt{2k}+g(k)\right)^2-\sqrt2\,\left(\sqrt{2k}+g(k)\right)^{3/2}\\
    =k+\sqrt{2k}g(k)+\frac12g(k)^2-2^{5/4}k^{3/4}+o\left(k^{3/4}\right).
  \end{multline*}
  This implies first $g(k)=O\left(k^{1/4}\right)$, and then $g(k)=2^{3/4}k^{1/4}+h(k)$ with
  $h(k)=o\left(k^{1/4}\right)$.
  Now the left-hand side of~(\ref{eq:target}) is
  \begin{multline*}
    \frac12t^*(k)^2-\sqrt2\,t^*(k)^{3/2}=\frac12\left(\sqrt{2k}+2^{3/4}k^{1/4}+h(k)\right)^2-\sqrt2\,\left(\sqrt{2k}+O\left(k^{1/4}\right)\right)^{3/2}\\
    =k+2\sqrt{2k}\,h(k)+O\left(k^{1/2}\right).
  \end{multline*}
  Now $h(k)=O(1)$ by~(\ref{eq:target}), and the result follows.
\end{proof}
Theorem~\ref{thm:asymptotics}(\ref{asymptotics:item_i}) will be proved by combining
Lemmas~\ref{lem:psi},~\ref{lem:f_asym} and~\ref{lem:tstar_asymp}. For the second part, we observe
that for even $n$, say $n=2k$, \Cref{thm:main_result} implies $\phi(n)=\binom{n}{k}-\psi(k)$,
and then with Theorem~\ref{thm:asymptotics}(\ref{asymptotics:item_i}),
$\phi(n)=\binom{n}{k}-\frac12n^{3/2}-\frac{1}{\sqrt2}n^{5/4}+O(n)$. For odd $n$, say $n=2k-1$, we
have to work a bit harder, because \Cref{thm:main_result} implies only
$\phi(n)=\binom{n}{k}-s$, where $s$ is the largest integer less than $k^2$ which is not in
$\Sigma(k)\cup\Sigma(k-1)$.
\begin{lemma}\label{lem:non_shadow_size}
 If $t=j^*(k-1)+1$ and $s=t(k-1)-\binom{t}{2}-1$, then $s\not\in\Sigma(k)\cup\Sigma(k-1)$. Consequently, for $n\in\{2k,2k-1\}$,
  \[\binom{n}{k}-t(k-1)+\binom{t}{2}+1\not\in S(n).\]
\end{lemma}
\begin{proof}
  We prove this by checking that $s$ is in the gap between $\sigma(t-1,k)$ and $\sigma(t,k)$ and
  also in the gap between $\sigma(t-1,k-1)$ and $\sigma(t,k-1)$. It follows from
  $s<t(k-1)-\binom{t}{2}$ that all the elements of
  $\sigma(t,k)\cup\sigma(t,k-1)$ are larger than $s$. To conclude the proof it is
  sufficient to verify the inequality $\left(t-1\right)k<s$. From $t-1=j^*(k-1)$ and the definition
  of $j^*$, we obtain $\binom{t+1}{2}<k-1$, hence
  \[\left(t-1\right)k=t(k-1)+t-k<t(k-1)+t-\binom{t+1}{2}-1=s.\qedhere\]
\end{proof}

\begin{proof}[Proof of Theorem~\ref{thm:asymptotics}]
  To prove part~(\ref{asymptotics:item_i}), we combine Lemmas~\ref{lem:psi},~\ref{lem:f_asym} and~\ref{lem:tstar_asymp}:
  \[\psi(k)=t^*k-f(t^*)-1=\left(\sqrt{2k}+\sqrt[4]{8k}\right)k+O(k)=\sqrt2k^{3/2}+\sqrt[4]{8}k^{5/4}+O(k).\]
  For part~(\ref{asymptotics:item_ii}), let $k=\lceil n/2\rceil$ and set
  $m_1=\binom{n}{k}-\psi(k)-1$. By part~(\ref{asymptotics:item_i}),
  \[m_1=\binom{n}{k}-(\sqrt2+o(1))k^{3/2}=\binom{n}{k}-\left(\frac12+o(1)\right)n^{3/2},\]
  and by \Cref{thm:main_result} together with the definition of $\psi(k)$, $\left[1,m_1\right]\subseteq S(n)$.
  For the other direction, let $t=j^*(k-1)+1$ and $m_2=\binom{n}{k}-t(k-1)+\binom{t}{2}+1$. By
  Lemma~\ref{lem:non_shadow_size}, $m_2\not\in S(n)$, and by Lemma~\ref{lem:jstar},
  \[m_2=\binom{n}{k}-\left(\sqrt{2k}+O(1)\right)(k-1)+O(k)=\binom{n}{k}-\left(\frac12+o(1)\right)n^{3/2}.\qedhere\]  
\end{proof}

\section{Open problems}\label{sec:open}
As mentioned in the introduction, our main result is a characterization of the possible sizes of
families in $B_n$ which are maximal with respect to the property of not containing a pair of sets,
one of which is contained in the other. In~\cite{Gerbner2012}, the saturation number
$\operatorname{sat}(n,P)$ for a poset $P$ has been introduced as the smallest size of a saturated
$P$-free family in $B_n$, and these numbers have been studied for various posets $P$,
see~\cite{Keszegh_2021} for an overview. For $P=C_{k+1}$ (a chain of length $k+1$),
it was shown in~\cite{Gerbner2012} that there is a value $\operatorname{sat}(k)$ such that
$\operatorname{sat}(n,C_{k+1})=\operatorname{sat}(k)$ for all sufficiently large $n$, and it is
known~\cite[Theorem 4]{Morrison_2014} that $\operatorname{sat}(k)=2^{(1+o(1))ck}$ for some $c$ with $\frac12\leq
c<1$. The largest size of a $C_{k+1}$-free family, $\operatorname{ex}(n,C_{k+1})$ is the sum of the
$k$ largest binomial coefficients. In analogy to \Cref{q:MAC_sizes}, it might be interesting to
investigate the following.
\begin{problem}
  For which integers $m$ with $\operatorname{sat}(k)\leq m\leq\operatorname{ex}(n,C_{k+1})$ does there exist a saturated
  $C_{k+1}$-free family of size $m$?
\end{problem}

In Theorem~\ref{thm:shadow_spectrum}, we characterize the shadow spectrum $\sigma(t,k)$ for
$t\leq k+1$. It is an interesting problem to study the set $\sigma(t,k)$ for general $t$. Some
initial results in this direction have been obtained in~\cite{Leck1995}. Let
$F(t,k)$ be the minimum shadow size of a $t$-family of $k$-sets, so that
$\sigma(t,k)\subseteq[F(t,k),tk]$. In~\cite{Leck1995}, the following sufficient conditions for
$s\not\in\sigma(s,t)$ are established:
\begin{itemize}
\item If there is an integer $a>k$ with $F(t-1,k)\leq\binom{a}{k-1}<s\leq F(t-1,k)+k-2$ then
  $s\not\in\sigma(t,k)$.
\item For an integer $s$ with $F(t,3)\leq s\leq 3t$, $s\not\in\sigma(t,3)$ if and only if
  $s=\binom{a}{2}+1$ for some integer $a\geq 4$ with $\binom{a}{3}-a+4\leq t\leq\binom{a}{3}$. 
\item For an integer $s$ with $F(t,4)\leq s\leq 4t$, each of the following three conditions implies $s\not\in\sigma(t,4)$:
  \begin{enumerate}[(i)]
  	
  \item $s=\binom{a}{3}+1$ for some integer $a\geq 5$ with $\binom{a}{4}-2a+9\leq
    t\leq\binom{a}{4}-a+4$,
  \item $s\in\left\{\binom{a}{3}+1,\,\binom{a}{3}+2\right\}$ for some integer $a\geq 5$ with $\binom{a}{4}-a+5\leq
    t\leq\binom{a}{4}$,
  \item $s=\binom{a}{3}+\binom{b}{2}+1$ for some integers $a>b\geq 4$ with $\binom{a}{4}+\binom{b}{3}-b+4\leq
    t\leq\binom{a}{4}+\binom{b}{3}$.
  \end{enumerate}
\end{itemize}
It would be interesting to prove the following generalization.
\begin{conjecture}[Vermutung 2.3 in~\cite{Leck1995}]
  Let $s$ be an integer and suppose there are integers $a_k>a_{k-1}>\dots>a_r>r\geq 3$ with
  \[F(t-1,k)\leq\binom{a_k}{k-1}+\binom{a_{k-1}}{k-2}+\dots+\binom{a_r}{r-1}<s\leq F(t-1,k)+r-2.\]
  Then $s\not\in\sigma(t,k)$.
\end{conjecture}

\section*{Acknowledgment}
We are grateful to two anonymous referees for their constructive comments which helped us to improve the
presentation of the arguments, and for pointing out references which were useful for putting our
results into the context of related work.

\printbibliography

\end{document}